\documentclass[review,onefignum,onetabnum,final]{siamart190516}

\usepackage{graphicx,epstopdf}
\usepackage{subfig}
\usepackage{enumitem} 
\usepackage{tikz}
\usepackage{titlesec}

\usepackage{breqn} 
\usepackage{fancyhdr}
\usepackage{algorithm}
\usepackage{algorithmic}
\usepackage{lipsum}
\usepackage{amsfonts}

\newtheorem{assumption}{Assumption}
\newtheorem{remark}{Remark}

\usepackage{amssymb}
\usepackage{bbm}
\usepackage[numbers,sort&compress]{natbib}
\usepackage{hyperref}
\def\equationautorefname~#1\null{ (#1)\null}
\hypersetup{
  colorlinks = true,
  citecolor = blue
  }
  
\newcommand{\numsamples}{K}
\newcommand{\td}{\tilde}
\newcommand{\mbb}{\mathbb}
\newcommand{\mcal}{\mathcal}
\newcommand{\opn}{\operatorname}
\newcommand{\hsp}{\hspace{0.1cm}}
\newcommand{\hspB}{\hspace{0.3cm}}
\newcommand{\pd}{\partial}
\newcommand{\lan}{\left\langle}
\newcommand{\ran}{\right\rangle}

\newcommand{\dimLarge}{M}
\newcommand{\dimSystem}{Q}

\newcommand{\tstar}{t}

\newcommand{\calibMani}[2]{\mcal M_{\opn{calib}}(#2)}
\newcommand{\calibManiApprx}[2]{\mcal M_{\opn{calib},M}(#2)}

\newcommand{\tref}{t_{\opn{ref}}}
\newcommand{\tdiscrete}{\{t_l\}_{l=1,\dots,K}}
\newcommand{\trefGen}[1]{t_{\opn{ref(#1)}}}
\newcommand{\refGen}[1]{\opn{ref}(#1)}

\newcommand{\xMax}{x_{\opn{max}}}
\newcommand{\xMin}{x_{\opn{min}}}

\newcommand{\snapMat}{\mcal S}
\newcommand{\snapMatCalib}{\mcal S_{\opn{calib}}}
\newcommand{\snapMatCalibSub}[1]{\mcal S_{\opn{calib},#1}}

\newcommand{\featureFV}[1]{z_{M,#1}}

\newcommand{\sol}[1]{u(\cdot,#1)}

\newcommand{\solCali}[1]{g}

\newcommand{\solFV}[1]{u_M(\cdot,#1)}

\newcommand{\vecFV}[1]{U_M(#1)}
\newcommand{\vecFVCalib}[1]{U_{\opn{calib,M}}(#1)}

\newcommand{\spTransf}[2]{\varphi(\cdot,#2)}

\newcommand{\spTransfGrid}[2]{\varphi_M(\cdot,#2)}

\newcommand{\tspace}{D}

\newcommand{\hcoeff}{\omega}

\usetikzlibrary{decorations.pathreplacing}
\usetikzlibrary{decorations.shapes}
\usetikzlibrary{calc}

\tikzset{decorate sep/.style 2 args=
{decorate,decoration={shape backgrounds,shape=circle,shape size=#1,shape sep=#2}}}

\title{Data-Driven Snapshot Calibration via Monotonic Feature Matching\thanks{Submitted to the editors xxxx\funding{N.S and P.B are supported by the German Federal Ministry for Economic Affairs and Energy (BMWi) in the joint project "MathEnergy - Mathematical Key Technologies for Evolving Energy Grids", sub-project: Model Order Reduction (Grant number: 0324019B). J.G is supported by DFG grant SFB TRR 154, project C05.}}}

\author{Neeraj Sarna\thanks{Corresponding author, Max Planck Institute for Dynamics of Complex Technical Systems, Sandtorstr 1, 39106, Magdeburg, Germany, \email{sarna@mpi-magdeburg.mpg.de}.}
\and Jan Giesselmann\thanks{Department of Mathematics, Technical University of Darmstadt, Darmstadt, 64293, Germany, \email{giesselmann@mathematik.tu-darmstadt.de.}}
\and Peter Benner\thanks{Max Planck Institute for Dynamics of Complex Technical Systems, Sandtorstr 1, 39106, Magdeburg, Germany, and 
Faculty of Mathematics, Otto Von Guericke University Magdeburg, Gustav-Adolf-Str., 39106, Magdeburg 
\email{benner@mpi-magdeburg.mpg.de}.}}
\date{%
}

\headers{Snapshot Cabliration Via Monotonic Feature Matching}{Neeraj Sarna, Jan Giesselmann, and Peter Benner}

\usepackage{amsopn}
\nolinenumbers
\begin{document}
\maketitle
\begin{abstract}
Snapshot matrices of hyperbolic equations have a slow singular value decay, resulting in  inefficient reduced-order models. We develop on the idea of inducing a faster singular value decay by computing snapshots on a transformed spatial domain, or the so-called snapshot calibration/transformation. We are particularly interested in problems involving shock collision, shock rarefaction-fan collision, shock formation, etc. For such problems, we propose a realizable algorithm to compute the spatial transform using monotonic feature matching. We consider discontinuities and kinks as features, and by carefully partitioning the parameter domain, we ensure that the spatial transform has properties that are desirable both from a theoretical and an implementation standpoint. We use these properties to prove that our method results in a fast $m$-width decay of a so-called calibrated manifold. A crucial observation we make is that due to calibration, the $m$-width does not only depend on $m$ but also on the accuracy of the full order model, which is in contrast to elliptic and parabolic problems that do not need calibration. The method we propose only requires the solution snapshots and not the underlying partial differential equation (PDE) and is therefore, data-driven. We perform several numerical experiments to demonstrate the effectiveness of our method.

\end{abstract}
\section{Introduction}
Several problems of practical interest are modeled using parameterized PDEs of the form 
\begin{gather}
\mcal L u(x,\mu) = 0\hspB \forall (x,\mu)\in\Omega\times\tspace. \label{evolve prim}
\end{gather}
Here, $\mcal L$ is some differential operator, $\mu\in\tspace$ is some parameter which can encode, for example, different material properties, and $x\in\Omega\subset \mbb R^d$ is a space point. We refer to the book \cite{RBBook} for an elaborate discussion on different parameterized PDEs. Note that $\tspace$ can contain time and in the model problem that we consider later, it is indeed the time domain. Nevertheless, the present discussion applies to general parameter domains. Often, an exact solution to the above problem is unavailable and one seeks an approximation $u(\cdot,\mu)\approx u_\dimLarge(\cdot,\mu)$ in a finite-dimensional space $X_M$ spanned by some basis $\{\phi_i\}_{i=1,\dots,M}$. The approximation $\solFV{\mu}$ is what we refer to as the full-order model (FOM). We assume that $X_M\subset L^2(\Omega)$. 

In a multi-query setting, where a solution is required at several different parameter instances, computing a FOM is computationally expensive and infeasible. Reduced-order models (ROMs) aim to reduce this cost by splitting the solution algorithm into an online-offline phase. A broad description of these two phases is as follows---see \cite{PeterReview} for further details. First, in the offline phase, one computes a snapshot matrix $\snapMat\in\mbb R^{M\times K}$ given as
 $$\snapMat :=\left(\vecFV{\mu_1},\dots,\vecFV{\mu_\numsamples}\right),$$
where $\vecFV{\mu}\in\mbb R^\dimLarge$ is a vector containing all the degrees of freedom of $\solFV{\mu}$ i.e., $\left(\vecFV{\mu}\right)_j := \lan \phi_j,\solFV{\mu}\ran_{L^2(\Omega)}$ where $\{\phi_j\}_j$ is a set of basis functions for $X_M$. The parameters $\{\mu_i\}_{i=1,\dots,K}$ can be chosen uniformly, randomly, or using a greedy procedure based on an a-posteriori error indicator \cite{Veroy,MarioOpDIEM,DoubleGreedy,SamplingLattice}. 

In the online phase, one approximates $\vecFV{\mu}$ in the span of the first $m$ left singular vectors of $\snapMat$, or the so-called Proper-Orthogonal-Decomposition (POD) modes of $\snapMat$. We collect these vectors in the matrix $\mcal U_{m}(\snapMat)$ and with $U^{\opn{red}}_m(\mu)$ we represent an approximation to $\vecFV{\mu}$ in $\opn{range}(\mcal U_{m}(\snapMat))$. The online phase is efficient only if any given error tolerance of practical interest 
\begin{gather}
 \|U^{\opn{red}}_m(\mu)-\vecFV{\mu}\|_2\leq \texttt{TOL}, \label{error ROM}
 \end{gather}
 can be achieved  with a sufficiently small value (preferably $\ll M$) of $m$.

At least empirically, the singular value decay rate of the snapshot matrix is a good indicator of the decay rate of the error in \eqref{error ROM}; see \cite{RBBook,sPOD,RowleyPODFluids}. Let $\sigma_i(\snapMat)$ denote the $i$-th singular value of $\snapMat$. Then, for all $i\in\{1,\dots,\numsamples\}$, we find 
\begin{gather}
 \|\vecFV{\mu_i}-\Pi_{\opn{range}(\mcal U_m(\snapMat))}\vecFV{\mu_i}\|_2\leq   \|\snapMat-\mcal U_m(\snapMat)\mcal U_m(\snapMat)^T\snapMat\|_F = \underbrace{\sqrt{\sum_{i=m+1}^K\sigma_i(\snapMat)^2}}_{=:\Xi_{m}(\snapMat)}.  \label{error POD}
\end{gather}
Above, $\|\cdot\|_F$ represents the Frobenius norm, $\Pi_{\square}$ represents an orthogonal projection operator with $\square$ being a place holder for some finite-dimensional space, and $(\cdot)^T$ represent the transpose of a matrix. If $\{\mu_i\}_{i=1,\dots,K}$ is sufficiently dense in $\tspace$ then, with the above relation, we expect the error in \eqref{error ROM} to decay at a similar rate as $\Xi_{m}(\snapMat)$. 

For hyperbolic problems, there is ample numerical evidence (also provided by the current article) supporting that $\Xi_{m}(\snapMat)$ decays slowly resulting in an inefficient ROM \cite{Cagniart2019,sPOD,nonino2019overcoming,Benjamin2018model}. Therefore, the first step toward developing an efficient ROM is to induce a faster singular value decay in the snapshot matrix, or to so-called calibrate the snapshot matrix. Following the works in \cite{Cagniart2019,Welper2017,sPOD}, we perform calibration by computing snapshots on a transformed domain. This results in a calibrated snapshot matrix that reads
\begin{equation}
\begin{gathered}
\snapMatCalib := \left(\vecFVCalib{\mu_1},\dots,\vecFVCalib{\mu_\numsamples}\right),\\
\text{where}\hspB \left(\vecFVCalib{\mu}\right)_j:=\lan \phi_i, u_M(\spTransfGrid{.}{\mu},\mu)\ran_{L^2(\Omega)}. \label{transformed snapshot}
\end{gathered}
\end{equation}
Above, $\varphi_M(\cdot,\mu):\Omega\to\Omega$ is a spatial transform that satisfies
\begin{equation}
\begin{aligned}
&\text{(P1)} \hsp \spTransfGrid{t}{\mu} \text{ is a homeomorphism},\\
&\text{(P2)} \hsp \|D_x\spTransfGrid{t}{\mu}^{-1}\|_{L^\infty(\Omega)},  \|D_x\spTransfGrid{t}{\mu}\|_{L^\infty(\Omega)}\leq \mcal K_1,\label{prop phi}
\end{aligned}
\end{equation}
where, $\mcal K_1 > 1$ is a user-defined constant and $D_{\square}$ denotes a weak-derivative with $\square$ being a place holder for a variable. We can think of $\varphi_M$ as a way of artificially introducing the desired regularity in the snapshots along the parameter domain, which eventually results in a fast singular value decay. For further clarification, we refer to the numerous examples and arguments in \cite{Cagniart2019,Welper2017,RimTR} and to the later sections of our work. The properties (P1) and (P2) are desirable from both a theoretical and a numerical implementation standpoint. They will be particularly helpful in studying the $m$-width of a so-called calibrated manifold defined below. Later sections provide further elaboration.

Note that snapshot calibration is an offline step. In the online phase, we can use the POD modes of $\snapMatCalib$ to approximate $\vecFVCalib{\mu}$ and then recover an approximation to $\vecFV{\mu}$ using $\varphi_M(\cdot,\mu)^{-1}$, or its approximation.  
Development of a PDE-based online algorithm that is stable, efficient and competitive with  finite-element/volume/difference type approximations is another challenging task and we plan to tackle it in the future---preliminary, but noteworthy, work in this direction can be found in \cite{Cagniart2019,SarnaGrundel2020,MATS,Nair}.

We propose a data-driven and feature-matching-based algorithm to compute $\varphi_M$ that satisfies (P1) and (P2). Let us elaborate on what we mean by feature matching. A feature is either a discontinuity or a kink (defined precisely later) in a snapshot $\solFV{\mu}$, and with $z_M(\mu)$ we represent its spatial location. We want the feature locations in $u_M(\spTransfGrid{.}{\mu_i},\mu_i)$ to coincide with those in some reference snapshot $\solFV{\mu_{\opn{ref}}}$ i.e.,
\begin{gather}
\varphi_M(z_M(\mu_{\opn{ref}}),\mu_i) = z_M(\mu_i),\hspB\forall i\in \{1,\dots,\numsamples\}. \label{match feature}
\end{gather}
We extend $\varphi_M(\cdot,\mu_i)$ to $\Omega$ by piecewise linear interpolation. We allow for multiple-features, feature interaction and feature formation. In order to deal with these cases, we propose an adaptive selection of the reference snapshot $\solFV{\mu_{\opn{ref}}}$ such that (P1) and (P2) are satisfied. In \Cref{sec: feature match} we discuss feature matching in further detail. Note that due to its data-driven nature, our algorithm treats all discontinuities the same i.e., it does not differentiate between shocks and contact discontinuities.

Most of the previous model-order reduction methods for hyperbolic equations were restricted to either periodic or extrapolated boundary conditions---for instance, see \cite{sPOD,RimTR,RimOT,Sapsis2018,FrozenROM}. The reason being that these works relied on either a (or multiple) spatial shift, a Lie group action, or an optimal transport map, all of which have some restrictions on the boundary conditions. We show that general time-dependent boundary conditions are naturally included in the feature matching framework by defining the boundary points as additional features. The numerical experiments included in \Cref{sec: num exp} showcase that our method works well for time-dependent boundary conditions.

In an abstract sense, an approximation of $\vecFVCalib{\mu}$ in the POD modes of $\snapMatCalib$ is a linear approximation of the so-called calibrated snapshot manifold defined as
\begin{gather}
\mcal M_{\opn{calib},M}(\tspace) := \{\Pi_{X_M}u_M(\varphi_M(\cdot,\mu),\mu)\hsp :\hsp \mu\in\tspace\}.\label{calib mani apprx}
\end{gather}
A linear approximation can be accurate only if the $m$-width of $\calibManiApprx{}{\tspace}$ decays fast. We prove that this is indeed the case for the calibrated manifold resulting from feature matching.
We provide a bound for the $m$-width of  $\mcal M_{\opn{calib},M}(\tspace)$ in case the FOM is a finite volume (FV) scheme. Our bound depends explicitly on both $m$ and $M$. To the best of our knowledge, no earlier works provide such a bound, making our work the first of its kind that provides a theoretical justification for feature matching. Note that, compared to the definition of the calibrated manifold proposed in \cite{Cagniart2019}, our definition is closer to what is actually used in practice---our definition uses the FOM whereas the one in \cite{Cagniart2019} uses the exact solution of the evolution equation \eqref{evolve prim}. The bounds on the $m$-width are discussed in detail in \Cref{sec: width decay}. 

We propose to match both kinks and discontinuities.
Usually, one would only match discontinuities---see for instance \cite{Cagniart2019,Welper2017}. This could be because (i) kinks get smeared out due to numerical dissipation and go undetected, or (ii) because, despite the kinks being detectable, they are not included in the set of features. For the first case, we show that, due to smearing, the FOM has sufficient regularity to ensure a fast $m$-width decay. For the second case, we show that matching both kinks and discontinuities provides a better calibration than only discontinuity matching. Precisely, in \Cref{sec: width decay}, we prove that both kink and discontinuity matching results in a calibrated manifold with an $m$-width that is $\mcal O(m^{-2})$, which is $\mcal O(m^{-1})$ times better than what only discontinuity matching offers.  To summarize, we establish that if kinks are detectable, then it is advantageous to include them in the feature set.

In \Cref{sec: num exp}, we perform several numerical experiments showcasing the effectiveness of our method. Mindful of the above discussion, we consider highly accurate approximations in $X_M$ where both kinks and discontinuities can be identified. For this reason, we consider the best-approximation in $X_M$ and show that kink and discontinuity matching results in a fast singular-value decay and that both kink and discontinuity matching is better than only discontinuity matching.

Our method is explicit in the sense that we explicitly compute the feature locations and match them. In the context of model-order reduction, explicit methods have been used before (see \cite{RBHyp,Constantine}), but never for problems involving multiple-features and feature interaction. Rather than using an explicit method, one can also solve an optimization problem and expect the features to be matched implicitly \cite{Welper2017,Nair}. The following reasons motivated our choice of an explicit method. Firstly, the optimization problem in implicit methods is (usually) non-convex and non-linear. If the samples $\{\mu_i\}$ are not chosen carefully, then the minimization problem can get stuck in sub-optimal local minima, resulting in a $\snapMatCalib$ with a slow singular value decay. Secondly, explicit methods rely on shock tracking/identifying techniques that are well-studied for hyperbolic problems \cite{DGBook}. Thirdly, in explicit methods, it is easier to quantify (at least empirically) the error in identifying the true feature location, which is helpful in quantifying the $m$-width decay rate. Lastly, with an access to feature locations, it easier to satisfy (P1) and (P2), which otherwise have to be included as constraints in the optimization problem. To the best of our knowledge, none of the implicit methods can impose such constraints.

We mention that apart from snapshot calibration, in the context of hyperbolic equations, other strategies to construct an accurate approximation space include online adaptivity of basis \cite{Benjamin2018model,LaxPairs}, embedding of the solution manifold in the Wasserstein metric space \cite{Metric2019} and the use of auto-encoders \cite{KevinAuto}. Comparison of the approximation space resulting from snapshot calibration to these other works is an interesting question in its own right and we plan to tackle it in the future.

\section{Feature Matching}\label{sec: feature match}
As a model problem, we interpret time as a parameter and consider the time-dependent hyperbolic conservation law in one space dimension given by
\begin{equation}
\begin{gathered}
\pd_t u(x,t) + \pd_x f(u(x,t)) = 0,\hspB\forall (x,t)\in \Omega\times \tspace,\hspB u(x,t=0) = u_0(x)\hspB\forall x\in \Omega,\\
u(x,t) = \mcal G(x,t),\hspB \forall (x,t)\in\pd\Omega\times\tspace.
\label{evolution eq}
\end{gathered}
\end{equation}
Above, $\tspace := [0,T]$ is the time-domain with some final time $T>0$, $u_0$ is the initial data and $\mcal G$ is some (given) boundary data. We interpret the boundary conditions in a weak-sense as described in \cite{BCHyp}. The solution vector $u$ maps $\Omega\times\tspace$ to $\mbb R^\dimSystem$ and $f:\mbb R^\dimSystem\to\mbb R^\dimSystem$ is a so-called flux function, where we allow $\dimSystem\geq 1$.
We restrict to a one-dimensional spatial domain with $\Omega:=(\xMin,\xMax)\subset\mbb R$. We consider a FV approximation space $X_M$ where
we partition $\Omega$ into $M$ sub-intervals of the same size $\Delta x = (\xMax-\xMin)/M$ i.e.,
\begin{gather}
    \Omega = \bigcup_{i=1}^M \mcal I_i,\hspB |\mcal I_i| = \Delta x. \label{partition space}
\end{gather}
For notational simplicity, we consider a uniform spatial grid---an extension to non-uniform grids is straightforward.

For notational simplicity, we restrict our discussion to scalar problems i.e., $\dimSystem=1$ in \eqref{evolution eq}. An extension to systems follows by applying the proposed method to every component of the solution vector. 
We find $\varphi_M$ such that the feature locations in $u_M(\spTransfGrid{t}{\tstar_k},\tstar_k)$ match to those in some reference snapshot $\solFV{t_{\opn{ref}}}$. 
The methodology used to compute $\varphi_\dimLarge$ drives the choice for $\solFV{\tref}$. For the present discussion, we choose 
\begin{gather}
t_{\opn{ref}} = 0. \label{ref time}
\end{gather} 
The motivation behind our choice becomes clear as we proceed. First, we define the notion of a feature. Note that the definition implicitly assumes that the exact solution has a finite number of features, a reasonable assumption for most problems of practical interest.

\begin{definition}[Feature] \label{def: feature}
A feature is either a discontinuity or a kink in the solution. For any $t\in \tspace$, let there be $p(t)\in\mbb N$ of such features. With $z_i(t)$ we represent the $i$-th feature location in $\sol{t}$. Furthermore, with $\featureFV{i}(t)$ we denote an approximation to $z_i(t)$ computed using $\solFV{t}$. Assuming that between the locations of discontinuities $\sol{t}$ has a weak derivative, we define a kink location as a space point where this weak derivative is discontinuous. Furthermore, we define the boundary points of $\Omega$ as two additional feature locations i.e., 
\begin{gather}
z_0(t) = z_{M,0}(t) =  \xMin,\hspB z_{p(t)+1}(t) =z_{M,p(t)+1}(t)=\xMax. 
\end{gather}
 Without loss of generality, we assume the ordering 
 $$\featureFV{0}(t) < \featureFV{1}(t)<\dots < \featureFV{p(t)+1}(t). $$ 
\end{definition}

We want to match the same type of features i.e., kinks with kinks and discontinuities with discontinuities. To distinguish between these two types of features, we associate an identifier with a feature location and define it in the following.

\begin{definition}[Identifier] \label{def gamma}
The identifier $\Gamma:\Omega\to\{0,1\}$ acts on a feature location and returns zero or one depending on whether there is a discontinuity or a kink at that location, respectively. For convenience, we collect all the identifiers in a vector $\gamma_M(\tstar_k)\in\mbb R^{p(t)}$ defined as $\left(\gamma_M(\tstar_k)\right)_i = \Gamma(\featureFV{i}(\tstar_k)).$
\end{definition}

We ask the following question. For some $t\in\tdiscrete$, given a snapshot $\solFV{t}$ and a reference snapshot $\solFV{\tref}$, does there exist a $\varphi_M$ that satisfies (P1) and (P2) and, in the sense of \eqref{match feature}, matches the features between $u_M(\spTransfGrid{t}{\tstar},\tstar)$ and $\solFV{\tref}$? We show that the answer to this question is yes if the following three conditions are satisfied
\begin{equation}
\begin{gathered}
\text{(C1)}\hspB p(t) = p(\tref),\hspB \text{(C2)}\hspB \gamma_M(t) = \gamma_M(\tref),\\
\text{(C3)}\hspB \frac1{\mcal K_1}\leq  \frac{|\featureFV{i+1}(\tref) - \featureFV{i}(\tref)|}{|\featureFV{i+1}(t) - \featureFV{i}(t)|}\leq \mcal K_1  \quad \forall  {i\in \{0,\dots,p(t)\}}. \label{assume naive}
\end{gathered}
\end{equation}
Above, $\mcal K_1$ is the same as that defined in \eqref{prop phi}.
The conditions (C1) and (C2) imply that the two snapshots have the same number and the same types of features. Furthermore, relative to $\solFV{\tref}$, (C3) prevents the features in $\solFV{t}$ from either coming too close or from moving very far away from each other. One can interpret the conditions (C1)-(C3) as a way of measuring the similarity of a snapshot to the reference snapshot, and if similar, we can find a $\varphi_M$ between the two snapshots that satisfies (P1) and (P2). If (C1)-(C3) is satisfied, then we say that $\solFV{t}$ matches to $\solFV{\tref}$ and for convenience, represent the matching by the notation
\begin{align}
\text{(C1), (C2) and (C3)}\hspB\Leftrightarrow\hspB \solFV{t}\leftrightarrow\solFV{\tref}. \label{equiv relation}
\end{align}

\subsection{Construction of $\varphi_M$}
Assume that $\solFV{t}\leftrightarrow\solFV{\tref}$ then feature matching provides 
$$\varphi_M(\featureFV{i}(\tref),\tstar) = \featureFV{i}(t),\hspB\forall i\in \{0,\dots,p(t)+1\}.$$
Note that (C2) ensures that the above relation does not match discontinuities to kinks or vice-versa. Furthermore, including the endpoints of $\Omega$ as features implies that $\varphi_M(\pd\Omega,t) = \pd\Omega$.
To extend $\varphi_M(\cdot,t)$ to $\Omega$, we perform a piecewise linear
interpolation, which, for $i \in \{0,\dots,p(t)\}$ and $x\in [z_{M,i}(\tref),z_{M,i+1}(\tref)]$, provides
\begin{equation}
\begin{aligned}
\varphi_M(x,\tstar) = &\left(\frac{x-\featureFV{i}(\tref)}{\featureFV{i + 1}(\tref)-\featureFV{i}(\tref)}\right)\featureFV{i + 1}(\tstar)\\
& + \left(\frac{x-\featureFV{i + 1}(\tref)}{\featureFV{i}(\tref)-\featureFV{i + 1}(\tref)}\right)\featureFV{i}(\tstar). \label{def spTransf}
\end{aligned}
\end{equation}

Trivially, $\spTransfGrid{.}{t}$ is continuous upto the boundary with $\varphi_M(\pd\Omega,t) = \pd\Omega$, which, due the ordering of the features in \Cref{def: feature}, implies that $\spTransfGrid{}{t}$ is strictly increasing. Thus, $\spTransfGrid{.}{t}$ is a homeomorphism.  Furthermore, the following relation and (C3) provides (P2). For all $t\in\tdiscrete$ and ${i\in \{0,\dots,p(t)\}}$, we find
\begin{equation}
\begin{aligned}
\frac{1}{\mcal K_1} \leq \left.D_x\varphi_M(\cdot,t)\right\rvert_{(\featureFV{i}(\tref),\featureFV{i+1}(\tref) )} = &\frac{|\featureFV{i+1}(\tref) - \featureFV{i}(\tref)|}{|\featureFV{i+1}(t) - \featureFV{i}(t)|}\\
\leq &\mcal K_1 .\label{der spTransf}
\end{aligned}
\end{equation}

We elaborate on why it is desirable to have (P1) and (P2).
\begin{enumerate}
\item \textit{Onto property:} as mentioned in the introduction, eventually in an online phase we want to approximate the calibrated snapshot $\vecFVCalib{t}$ in span of the POD modes of $\snapMatCalib$. We expect such an approximation to be accurate if $\spTransfGrid{}{t}$ is an onto function. We also refer to the arguments made in \cite{Welper2017} and our analysis in \Cref{sec: width decay} indicating that the onto property is desirable. At least intuitively, the following example further elaborates on the desirability of the onto property. Suppose that the characteristics curves originating from $t=0$ pass through every point in $\Omega$ for some $t^*\in \tspace$. Then a $\varphi_M(\cdot,t^*)$ that is not onto, will discard some information in $\solFV{t^*}$, which is undesirable and inconsistent with the characteristics.
\item \textit{Invertibility:} the analysis in \Cref{sec: width decay} indicates that the invertibility of $\varphi_M(\cdot,t)$ is desirable.
\item \textit{Continuity and monotonicity:} continuity and monotonicity of $\spTransfGrid{t}{\tstar}$ ensure that, as compared to $\solFV{\tstar}$, no new discontinuities appear in $u_M(\spTransfGrid{t}{\tstar},\tstar)$.
For the same reason, $\spTransfGrid{t}{\tstar}^{-1}$ should also be continuous. Points (1)-(3) imply that $\spTransfGrid{t}{\tstar}$ should be a homeomorphism i.e., it should satisfy (P1).
\item \textit{Bounds on the derivatives:} the bound on the $m$-width, which we present later in \Cref{sec: width decay}, scales with $\|D_x\varphi_M(\cdot,\tstar)\|_{L^{\infty}(\Omega)}$ and $\|D_x\varphi_M(\cdot,\tstar)^{-1}\|_{L^{\infty}(\Omega)}$, which motivates (P2).
\end{enumerate}

\subsection{Open questions}
The above formulation leaves the following questions open. The rest of the article (tries) to answer them. 
\begin{itemize}
\item How to handle the cases where (C1)-(C3) are not satisfied?
\item How to determine the feature locations in practise?
\item Why does feature matching result in a fast singular value decay?
\end{itemize}

In relation to the first question, it is easy to violate (C1). Consider \Cref{fig: shock collision} that shows the time-trajectory of two discontinuities in an otherwise smooth function. At $t=T_0$, two discontinuities interact to form a single one, changing the value of $p(t)$ from two to one. We handle such cases by partitioning $\tdiscrete$ into subsets and choosing (different) suitable reference snapshots such that (C1)-(C3) is locally satisfied in each of the subsets. The details are discussed in \Cref{sec:ARSS}.

To answer the third question rigorously, we need decay estimates for the singular values of the calibrated snapshot matrix $\snapMatCalib$. Such estimates are unavailable, as yet. However, later (in \Cref{sec: width decay}), we prove that feature matching results in a fast $m$-width decay of the calibrated manifold defined in \eqref{calib mani apprx}. At least empirically, a fast $m$-width decay results in a fast singular value decay of the snapshot matrix. Our expectation is corroborated by the numerous numerical experiments (performed in \Cref{sec: num exp}) where we empirically establish a fast singular value decay in the calibrated snapshot matrix.

\begin{figure}
\centering
\begin{tikzpicture}
\draw[->] (-0.1,0) -- (-0.1,3);
\draw[->] (-0.1,0) -- (3,0);
\draw[black,line width=0.5mm] (0,0) -- (1,1);
\draw[black,line width=0.5mm] (2,0) -- (1,1);
\draw[black,line width=0.5mm] (1,1) -- (1.5,2);
\draw[dashed] (0,1) -- (2,1);
\draw[dashed] (0,2) -- (2,2);
\node[] at (-0.3,1) {$T_0$};
\node[] at (-0.3,2) {$T$};
\node[] at (-0.3,3) {$t$};
\node[] at (3,-0.15) {$x$};
\end{tikzpicture}
\caption{Time trajectory of two discontinuities that merge to form a single discontinuity.}\label{fig: shock collision}
\end{figure}
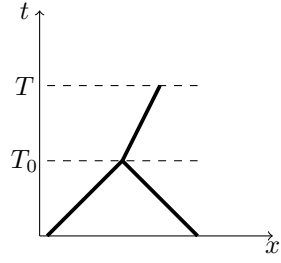
\subsection{Adaptive reference snapshot selection}\label{sec:ARSS} Recall the conditions (C1)-(C3) given in \eqref{assume naive}. A snapshot $\solFV{t_k}$ cannot be matched to $\solFV{\tref}$ if either of these three conditions are violated. To handle such cases, we partition $\tdiscrete$ into subsets containing subsequent time-instances. For each of these subsets, we find a different $\tref$ such that (C1)-(C3) is satisfied locally. The details are as follows. 

We start with introducing the following notation.

\begin{definition}[Time partitions]\label{time partitions}
We partition $\tdiscrete$ into $N \in \mbb N$ subsets (where $N$ will be an outcome of the snapshot selection algorithm). We denote the $i$-th subset by $[t]_i$. With $r(i)\in\mbb N$ we denote the number of elements in $[t]_i$, and with $\trefGen{i}$ we denote the first element of $[t]_i$, where $\refGen{i}$ is an index in $\{1,\dots,K\}$. Under this notation, $[t]_i$ reads 
$$[t]_i := \{\trefGen{i},\dots,t_{\opn{ref(i)} + r(i)-1}\}.$$
\end{definition}

\Cref{algo: part time} presents the reference snapshot selection algorithm. The algorithm starts with the initial data as the reference snapshot, compares it to the subsequent snapshots and, in case matching is not possible, updates the reference snapshot. In addition to checking (C1)-(C3), the algorithm enforces a lower bound on the minimum distance between the features. At least empirically, one observes that the error in computing a feature location (i.e., $|z_j(t) - \featureFV{j}(t)|$) is of the order of the grid-size $\Delta x$. Therefore, to have a reliable calibration we need
\begin{gather}
\mcal K_2\Delta x \leq \min_j|\featureFV{j+1}(t)-\featureFV{j}(t)|\hspB\text{where}\hspB \mcal K_2 \geq 2.  \label{lower bound dist}
\end{gather}

The output of the algorithm are the time-indices $\{\refGen{i}\}_{i=1,\dots,N}$ of the reference snapshots. With these time indices, we construct $[t]_i$ as $[t]_i = \{t_{\refGen{i}},\dots,t_{\refGen{i+1}-1}\}$. Furthermore, with the help of $[t]_i$, we split the snapshot matrix as
\begin{gather}
\snapMat = \left(\snapMat_1,\dots,\snapMat_N\right), \label{sub-matrix}
\end{gather}
where each of the sub-matrices $\snapMat_i\in\mbb R^{\dimLarge\times r(i)}$ contain the snapshots for all $t\in [t]_i$ and can be calibrated using feature matching. 

\begin{remark} We further elaborate on the importance of ensuring the lower-bound in \eqref{lower bound dist}. For proper calibration, the ordering of features observed in the numerical solution should be the same as that for the exact solution. At least empirically, we observe that the feature detection algorithm provides feature locations that are correct up to errors of size $\Delta x$.
Therefore, in our numerical experiments we do not match snapshots containing features that are closer than $2 \Delta x$ to any other snapshots i.e., we satisfy the lower-bound in \eqref{lower bound dist}.
\end{remark}

\subsection{Approximation space}
We discuss how to use the above splitting of the snapshot matrix to construct an approximation space for the calibrated snapshot $\vecFVCalib{t}$ defined in \eqref{transformed snapshot}. We first consider the time interval $D_i$, which is a continuous analogue of $[t]_i$, and is given as 
\begin{gather}
\tspace_i := [\trefGen{i},t_{\opn{ref}(i) + r(i)-1}]. \label{def Di}
\end{gather}
Let $\snapMatCalibSub{i}$ represent a calibration of $\snapMat_i$. In the online phase, for $t\in \tspace_i$, we approximate $\vecFVCalib{t}$ in the span of the first $m_i$ left singular-vectors of $\snapMatCalibSub{i}$ i.e., in $\opn{range}(\mcal U_{m_i}(\snapMatCalibSub{i}))$.

We now consider the time interval $d_i$, which is the gap between $D_i$ and $D_{i+1}$, and reads
\begin{gather}
    d_i := (t_{\refGen{i+1}-1},t_{\refGen{i+1}}). \label{def di}
\end{gather}
Since the snapshots $\solFV{t_{\refGen{i+1}-1}}$ and $\solFV{t_{\refGen{i+1}}}$ do not match, we need information from both $\snapMatCalibSub{i}$ and $\snapMatCalibSub{i+1}$ for an accurate approximation of $\vecFVCalib{t}$. Therefore, we consider the approximation space $\opn{range}(\mcal U_{m_i}(\snapMatCalibSub{i}))+ \opn{range}(\mcal U_{m_{i+1}}(\snapMatCalibSub{i + 1}))$. We summarize our above discussion.
\begin{enumerate}
    \item For $t\in \tspace_i$, approximate $\vecFVCalib{t}$ in $\opn{range}(\mcal U_{m_i}(\snapMatCalibSub{i}))$.
    \item For $t\in d_i$, approximate $\vecFVCalib{t}$ in the sum of $\opn{range}(\mcal U_{m_i}(\snapMatCalibSub{i}))$ and $ \opn{range}(\mcal U_{m_{i+1}}(\snapMatCalibSub{i + 1}))$.
\end{enumerate}

\begin{algorithm}[ht!]
\caption{Reference snapshot selection algorithm}
    \hspace*{\algorithmicindent} \textbf{Input: }$\snapMat$, $\mcal K_1$, $\mcal K_2$\\
    \hspace*{\algorithmicindent} \textbf{Output: }$\{\refGen{i}\}_{i=1,\dots,N}$
\begin{algorithmic}[1] \label{algo: part time}
\STATE Initialize with $N\gets 1$, $\refGen{N}\gets 1$ and $k \gets 1$
\STATE $\Delta_{\min} z(\trefGen{N})\gets \min_j|\featureFV{j + 1}(\trefGen{N})-\featureFV{j}(\trefGen{N})|$
\STATE $\Delta_{\min} z(t_k)\gets \min_j|\featureFV{j+1}(t_k)-\featureFV{j}(t_k)|$
\STATE Check whether the following conditions are satisfied: (C1)-(C3), $\Delta_{\opn{min}} z(\trefGen{N})  > \mcal K_2\Delta x$ and $\Delta_{\opn{min}} z(t_k)  > \mcal K_2\Delta x$.
\STATE If the above statement returns true, increment $k$ by on. Else, increment $N$ by one, change $\refGen{N}$ to $k$ and increase $k$ by one. 
\STATE Till $k\leq K$, repeat from line-2.

\end{algorithmic}
\end{algorithm}
\subsection{Relation to the previous works} To the best of our knowledge, only the works in \cite{WelperAdaptive,sPOD} propose a snapshot calibration technique for problems involving feature interaction and formation. We compare our method to both of these works. The authors in \cite{WelperAdaptive} propose a so-called transformed snapshot interpolation (TSI) to handle shock collision problems and it differs from the current work in the following ways. Firstly, authors use an implicit method (the method requiring a solution to an optimization problem, see the introduction) to find the transform $\varphi_M$. Secondly, authors partition the time-domain using a $hp$-finite element strategy, which does not rely on a reference snapshot selection.  Thirdly, it is unclear whether the transform $\varphi_M$ satisfies the properties (P1) and (P2) both of which, at least according to our analysis, are crucial. 

Our method differs from the shifted-POD approach (proposed in \cite{sPOD}) in the following sense. Firstly, shifted-POD is an iterative algorithm where each iteration calibrates a particular transport mode by shifting the spatial domain. Our spatial transform $\varphi_M$ takes care of all the transport modes in one step, avoiding the need for iterations. Secondly, the shift value computation in shifted-POD requires a significant user-interference and results from either a careful observation of the snapshot matrix $\snapMat$ or of its singular values. In comparison, after the snapshot matrix is computed, our method to compute $\varphi_M$ is automatic. Thirdly, the shifted-POD does not cater to time-dependent boundary conditions. Note that none of the above two works study the $m$-width decay of the calibrated manifold.

\section{Kolmogorov $m$-width decay}\label{sec: width decay}
In this section, we study the $m$-width of the calibrated manifold $\calibManiApprx{t}{\tspace_i}$ defined in \eqref{calib mani apprx}. Here, $D_i$ is the continuous analogue of $[t]_i$ defined in \eqref{def Di}. This section has two main highlights (i) the bound on the $m$-width does not only depend on the ROM dimension $m$ but also on the FOM dimension $M$, and (ii) for  sufficiently regular initial data $u_0$ and flux function $f$, the $m$-width decays fast with respect to $m$. Precisely, when the FOM is a FV solution, we show that 
\begin{gather}
    \delta_m(\calibManiApprx{t}{\tspace_i}) = \mcal O(m^{-\hcoeff}) + \mcal O( M^{-\frac{1}{2}}),\label{higligh decay}
\end{gather}
where the coefficient $\hcoeff$ is related to the regularity of $u_0$ and $f(u_0)$ between the features. Furthermore,
for any manifold $\mcal M:=\{h(\cdot,t)\hsp:\hsp t\in\tspace\} \subset L^2(\Omega)$ its $m$-width, denoted by $\delta_m(\mcal M)$, is defined as
\begin{gather}
\delta_m(\mcal M):= \inf_{\substack{\mcal V_m\subset L^2(\Omega)\\ \opn{dim}(\mcal V_m) = m}}\|h-\Pi_{\mcal V_m}h\|_{L^2(\Omega\times \tspace)}. \label{def width}
\end{gather}
The $M$-dependency of the $m$-width appearing in \eqref{higligh decay} is introduced via the transform $\varphi_M$, which we compute using the FOM. Note that for elliptic and parabolic problems, calibration is not needed resulting in only an $m$-dependent $m$-width \cite{CohenElliptic}. 

We now discuss the details of the result mentioned above. We restrict ourselves to a scalar conservation law i.e., $Q=1$ in \eqref{evolution eq}. Furthermore, we make the standard assumption that the flux function $f$ is at least $C^2$ and is strictly convex. Note that for $t\in d_i$, where $d_i$ is the gap between $D_i$ and $D_{i+1}$ and is as given in \eqref{def di}, calibration using feature matching is not possible and therefore, $\calibManiApprx{t}{d_i}$ is irrelevant. Furthermore, since we use snapshots from both $D_{i}$ and $D_{i+1}$ to approximate the solution inside $d_i$, we expect this approximation to be accurate.

We start with defining a few quantities and making some assumptions.
In the earlier sections, we considered a discrete space-time domain. For a large enough $M$, we expect the feature locations $z_i(t)$ to behave similar to the approximate feature locations $\featureFV{i}(t)$. This motivates the assumption that since for all $t,t^*\in[t]_1$, we have $\solFV{t}\leftrightarrow\solFV{t^*}$, we also have
\begin{assumption}
$
\sol{t}\leftrightarrow \sol{t^*},\hspB\forall t,t^*\in\tspace_i  \text{ for all } i=1,\dots,N$. 
\end{assumption}
Our results are the same for all the different $\tspace_i$. Therefore, we present our results on some representative $D_i$ that we denote by $D$ for brevity.
The above assumption allows us to define the following.
\begin{definition}[Calibrated manifold]
Similar to $\calibManiApprx{}{\tspace}$, define
\begin{gather}
    \calibMani{.}{\tspace}:=\{u(\varphi(\cdot,t),t)\hsp :\hsp t\in \tspace\}. \label{calib manifold}
\end{gather}
Above, $\varphi$ is the same as $\varphi_M$ defined in \eqref{def spTransf} but with $\featureFV{j}(t)$ replaced by the exact feature location $z_j(t)$. We can interpret the functions in $\calibMani{}{\tspace}$ as a continuous-in-space analogue of those in $\calibManiApprx{.}{\tspace}$.
\end{definition}
In the next definition, we partition the space-time domain using the time-trajectory of different feature locations.
\begin{definition}[Space-time partitioning]
Let the number of features in $\solFV{\tref}$ be $p_0$ i.e., $p(\tref=0) = p_0$. For $i\in \{0,\dots,p_0\}$, define
\begin{gather}
\Omega_i := (z_{i}(0),z_{i+1}(0)),\hspB \Omega_i^D := \{(x,t)\hsp :\hsp x\in (z_i(t),z_{i+1}(t)),\hsp t\in\tspace\}. \label{def Omegai}
\end{gather} Note that $\opn{clos}{\left(\Omega\right)} = \bigcup_{i=0}^{p_0}\opn{clos}{\left(\Omega_i\right)}$. 
\end{definition}
The main result of this section and its corollary are summarised below. The rest of the section proves this result. 

\begin{theorem}\label{theorem: decay rate}
The $m$-width of the calibrated manifold $\delta_m(\calibManiApprx{}{\tspace})$ is bounded by
\begin{equation}\label{bound approx mani}
\begin{aligned}
\delta_m(\calibManiApprx{}{\tspace})\leq &\delta_m(\calibMani{}{\tspace})\\
&+ \sup_{t\in\tspace}\|D_x\varphi_M(\cdot,t)^{-1}\|_{L^\infty(\Omega)} \|u_M-u\|_{L^2(\Omega\times D)}\\
&+
\|u_M\circ \varphi_M-\Pi_{X_M}u_M\circ \varphi_M\|_{L^2(\Omega\times D)}\\
&+\sqrt{\|u\|_{L^\infty(\tspace;BV(\Omega))} \|u\|_{L^2(\tspace;BV(\Omega))}}\\
&\times \sqrt{\max_{j}\|\featureFV{j}-z_j\|_{L^{\infty}(\tspace)}}\times \max(1,\|D_x\varphi_M\|_{L^\infty(\Omega\times \tspace)}).
\end{aligned}
\end{equation}
\end{theorem}

\begin{corollary}\label{coro: width decay}
Provided the following conditions hold
 \begin{enumerate}
     \item The feature identification procedure used for computing $\varphi_M$ satisfies
     \begin{gather*}
         \max_{j}\|\featureFV{j}-z_j\|_{L^{\infty}(\tspace)}= \mcal O(M^{-1}).
     \end{gather*}
     \item There exists $\hcoeff \geq 1$ so that for all $i\in\{0,\dots,p_0\}$ the flux function and the initial data satisfy
    \begin{equation}
    \begin{gathered}
    f\in C^{\hcoeff + 1},\hspB u_0\rvert_{\Omega_i}\in W^{\hcoeff,\infty}(\Omega_i). \label{assume data}
    \end{gathered}
    \end{equation}
    Here, $u_0$ refers to the initial data at the beginning of the corresponding time interval $D_j$. 
Furthermore, $W^{\hcoeff,\infty}$ represents the Sobolev-space of functions having $\omega$ weak derivatives in $L^\infty$.
\item For all $i\in \{0,\dots,p_0\}$, 
\begin{gather}
\sup_{(x,t)\in \Omega_i^D}\frac{1}{|\beta_i(x,t)|}< \infty\hspB \text{where}\hspB \beta_i(x,t) := 1 + \tstar f^{''}(u_0(x))D_x u_0(x). \label{classic sol}
\end{gather}
 \end{enumerate}
Then,  for a convergent FV approximation scheme, using $M$ equidistant cells, the $m$-width satisfies
\begin{gather}\label{width decay rate}
    \delta_m(\calibManiApprx{}{\tspace}) = \mcal O(m^{-\hcoeff}) + \mcal O(M^{-\frac{1}{2}}).
\end{gather}
\end{corollary}

\begin{remark}
Note that the boundedness of $\beta_i$ is equivalent to no shock forming on $\opn{clos}\left(\Omega^D_i\right)$.
\end{remark}
We make the following observations and conclusions from the above result.
\begin{enumerate}
\item The bound on the $m$-width given in \eqref{bound approx mani} is robust under the limit $m\to\infty$ and $M\to \infty$. 


\item All the terms on the right in \eqref{bound approx mani}, apart from $\delta_m(\calibMani{.}{\tspace})$, are $M$-dependent i.e., they depend on the accuracy of the full-order model.  
\item For $M$ large enough  and $m$ small enough, we expect the bound to be dominated by $\delta_m(\calibMani{}{\tspace}) $. 
\item For a constant $M$, as $m\to\infty$, the bound will stagnate at a $\mcal O(M^{-\frac{1}{2}})$ term. This means that as $m \to \infty$, the best approximation error of $u$ in the ROM space is of the same order of magnitude as $ \|\sol{t}-\solFV{t}\|_{L^2(\Omega)}$, where $\solFV{t}$ is the FOM. Recall that the best approximation error of a (discontinuous) BV-function in a FV approximation space is $\mcal O(M^{-\frac{1}{2}})$.

The practical take-away from this discussion is that it does not make sense to increase  $m$ beyond a certain limit i.e., it does not make sense to further increase $m$ when $\|\solFV{t}-u^{\opn{red}}_m(\cdot,t)\|_{L^2(\Omega)} $ and $ \|\sol{t}-\solFV{t}\|_{L^2(\Omega)}$ are of the same order of magnitude. Here, $u^{\opn{red}}_m$ represent a reduced-order approximation to $u$.
\item Note that for $u \in W^{1,\infty} (\Omega \times \tspace)$, which allows only for kinks and no discontinuities, the best approximation error of $u$ in the FV approximation space is $\mcal O(M^{-1})$. Similarly, the last term on the right hand side of \eqref{bound approx mani} can be improved to $\|u\|_{L^2(\tspace;W^{1,\infty}(\Omega))}M^{-1}$.
 \item The bound on the $m$-width in \eqref{bound approx mani} explains that an upper-bound on $\|D_x\varphi_M(\cdot,t)^{-1}\|_{L^\infty(\Omega)}$ and $\|D_x\varphi_M(\cdot,t)\|_{L^\infty(\Omega)}$ (i.e. (P2) given in \eqref{prop phi}) are desirable. 
\item The bound in \Cref{theorem: decay rate} and \Cref{algo: part time} suggests a compromise between small and large values of $\mcal K_1$---recall that $\mcal K_1$ is the user-defined constant appearing in the property (P2) given in \eqref{prop phi}. As $\mcal K_1$ increases, \Cref{algo: part time} generates smaller number of reference snapshots, resulting in a calibrated snapshot matrix with a fewer number of sub-matrices. We expect that, for a given approximation accuracy, this would result in a fewer number of POD modes used to approximate the calibrated snapshot. In contrast, $\mcal K_1$ scales the $\mcal O(M^{-1/2})$ part of the bound in \Cref{theorem: decay rate}, making it undesirable to choose a large $\mcal K_1$. Numerical experiments indicate that any choice of $\mcal K_1$ that is $\mcal O(1)$ is acceptable.
\end{enumerate}

\subsection{Proof of Theorem 3.1}
Triangle's inequality applied to the definition of $\delta_m(\calibManiApprx{}{\tspace})$ provides
\begin{equation}
\begin{aligned}
\delta_m(\calibManiApprx{}{\tspace})\leq &\delta_m(\calibMani{}{\tspace})\\
&+ \sup_{t\in\tspace}\|D_x\varphi_M(\cdot,t)^{-1}\|_{L^\infty(\Omega)}\underline{\|u_M-u\|_{L^2(\Omega\times D)}}\\
&+
\underline{\|u_M\circ \varphi_M-\Pi_{X_M}u_M\circ \varphi_M\|_{L^2(\Omega\times D)}}\\
&+ \|u\circ \varphi_M-u\circ\varphi\|_{L^2(\Omega\times D)} \label{bound1}\\
&=:A_1 + A_2 + A_3 +A_4.
\end{aligned}
\end{equation}
A bound for the different $A_i$'s is as follows.
\subsubsection{Bound for $A_2$ and $A_3$}
A bound for $A_2$ and $A_3$ follows from the approximation properties of a FV approximation space. 
The decay (in $M$) of $A_2$ is connected to the convergence of the underlying FOM, if $u$ is in $BV \setminus W^{1,\infty}$  then $A_2$ will behave as $\mcal O(M^{-1/2})$. Here, $BV(\Omega)$ is a space of real-valued functions with a finite total variation.
Due to the approximation properties of the FV approximation space we have
\[ A_3 \leq |\Omega| M^{-1/2} |u_M \circ \varphi_M |_{L^2(D,BV(\Omega))} =  |\Omega| M^{-1/2} |u_M |_{L^2(D,BV(\Omega))} . \]
Note that we have used the monotonicity of $\varphi_M$ in the equality above and that $|u_M |_{L^2(D,BV(\Omega))} \leq |u|_{L^2(D,BV(\Omega))}$ provided the FV scheme is total-variation-diminishing (TVD).

\subsubsection{Bound for $A_1$} Let $g(x,t) = u(\varphi(x,t),t)$, where $\varphi$ is as given in \eqref{calib manifold}.
Tracing the characteristics backwards from $\tstar$ to $0$, we have
\begin{gather}
g(x,\tstar) = u_0(\underbrace{(\opn{Id} + \tstar f'(u_0))^{-1}\varphi(x,\tstar)}_{=:X_{i}(\varphi(x,\tstar),\tstar)}) \hspB\forall (x,t)\in\Omega_i\times \tspace, \label{trace back}
\end{gather}
where $u_0$ is the initial data in \eqref{evolution eq}, $f$ is the flux-function in \eqref{evolution eq}, and $\Omega_i$ is as defined in \eqref{def Omegai}. Note that because the flux function is convex, while tracing the characteristics backwards in an entropy solution, they do not run into a shock. 
Using \eqref{trace back}, the following result quantifies the regularity of $g$. 
\begin{lemma} \label{lemma: decay estimate}
Provided \eqref{assume data} and \eqref{classic sol} hold, then $g\in L^{2}(\Omega;H^\hcoeff(\tspace))$. Here, $L^{2}(\Omega;H^\hcoeff(\tspace))$ denotes a Bochner space of $L^2$ functions defined over $\Omega$ with values in the Sobolev space $H^\hcoeff(\tspace)$. 
\end{lemma}
\begin{proof}
See \Cref{app: proof decay}. 
\end{proof}

With the regularity established in the above result, taking the linear space $\mcal V_m$ (appearing in \eqref{def width}) to be the span of first $m$-Fourier modes in $\tspace$, we can estimate the $m$-width as
\begin{gather}
\delta_m(\calibMani{.}{\tspace})\leq \|g-\Pi_{\mcal V_m} g\|_{L^2(\Omega\times\tspace)} = \mcal O(m^{-\hcoeff}).
\end{gather}
Note that the (un-calibrated) solution $\sol{t}$ rarely has the amount of regularity that $g(\cdot,t)$ does. In this sense, we can view calibration as a way of "artificially" introducing regularity to induce a fast $m$-width decay in the calibrated solution manifold.

Apart from the above result, a trivial but noteworthy case is when $g$ is time-independent. This results in $\calibMani{}{\tspace}$ consisting of a single function, which provides
\begin{gather}
\delta_m(\calibMani{}{\tspace})=0\hspB\forall m\geq 1.
\end{gather}
Indeed, $g$ is time-independent provided, for all $i\in \{0,\dots,p_0\}$, either of the following two conditions hold
\begin{equation}
\begin{aligned}
\text{(i) } &u_0|_{\Omega_i} \equiv u_{0,i} \text{ for some constant } u_{0,i}\in\mbb R,
\\
\text{(ii) }&X_i(\varphi(x,t),t)\text{ is independent of $t$}. \label{cond trivial}
\end{aligned}
\end{equation}
The first condition corresponds to the initial data being a constant inside $\Omega_i$, and the second one can result inside a rarefaction fan; see \Cref{app: rarefaction}.

\begin{remark}
The result in \Cref{lemma: decay estimate} highlights the advantages of aligning both kinks and discontinuities. By including kinks into the set of features we can hope that $u_0$ is $W^{2,\infty}$ between features which makes $g \in L^2(\Omega, H^2(D))$ possible, resulting in a $m$-width that is $\mcal O(m^{-2})$. However, if $u_0$ contains a kink that is not in the set of features then we expect $u_0$ is $W^{1,\infty}\setminus W^{2,\infty}$ between the features resulting in $g \in L^2(\Omega, H^1(D))\setminus L^2(\Omega, H^2(D)) $ and a $m$-width that is $\mcal O(m^{-1})$.
\end{remark}

\begin{remark}
One can match the discontinuities in the higher-order derivatives of $\solFV{t}$ and get a faster (than presented above) $m$-width decay rate---precisely, matching discontinuities in the $\hcoeff$-order derivative results in a $\hcoeff+1$-order decay in the $m$-width. However, numerically identifying the location of discontinuities in higher-order derivatives is difficult and cumbersome. As our numerical experiments indicate, for a sufficiently refined numerical approximation in $X_M$, kink identification is possible and for that reason, we do not consider higher-order derivatives.
\end{remark}

\subsubsection{Bound for $A_4$}
The estimate for $\|u\circ \varphi_M-u\circ\varphi\|_{L^2(\Omega\times D)}$ follows from the result below. The first part of the result is an extension of the result in \cite{Welper2017} to $L^2$-functions and exploits the density of smooth functions in the $BV$-space. In the second part, we use the explicit from of the spatial transform given in \eqref{def spTransf} to compute $\|\varphi-\varphi_M\|_{L^{\infty}(\Omega\times\tspace)}$. With the bound given in the second part, we again emphasize on the desirability of ensuring (P2).
\begin{lemma}\label{lemma: error feature}
The following relations hold true.
\begin{enumerate}
\item $\|u\circ \varphi-u\circ \varphi_M\|^2_{L^2(\Omega\times\tspace)}\leq  \|u\|_{L^\infty(\tspace;BV(\Omega))}\|u\|_{L^2(\tspace;BV(\Omega))} \|\varphi-\varphi_M\|_{L^{\infty}(\Omega\times\tspace)}$.
\item Let $\mcal K_1$ be the constant given in \eqref{prop phi}. Then, the error $\|\varphi-\varphi_M\|_{L^{\infty}(\Omega\times\tspace)}$ is bounded as 
\begin{gather}
\|\varphi-\varphi_M\|_{L^{\infty}(\Omega\times\tspace)}\leq \max(1,\mcal K_1) \max_{j}\|\featureFV{j}-z_j\|_{L^{\infty}(\tspace)}.
\end{gather}
\end{enumerate}

\end{lemma}
\begin{proof}
See \Cref{app: proof error feature}.
\end{proof}

\section{Feature Detection}\label{sec: detect features}
It is important to note that our calibration approach can be combined with any feature detection approach and that the feature location algorithm can be used as a black-box. In order to keep this article self-contained, we explain one specific approach which was also used in our numerical experiments. This specific approach is based on the more general idea that kinks are discontinuities in the derivative i.e., discontinuities and kinks can be detected by discontinuity detection schemes using the following three steps: (i) approximate the discontinuity locations, (ii) approximate the weak derivative $D_x\sol{t}$ and (iii) approximate the kink locations by applying the discontinuity detection algorithm to $D_x\sol{t}$. To realize such a method, we need a discontinuity detector for which several different methods can suffice. For example, one can detect discontinuities by training a neural network \cite{ShockNN}, using the convergence properties of FOM \cite{ShockSupConvg}, performing a multi-resolution-analysis (MRA) \cite{MRAdetect2014}, etc. 

For its ease of implementation and reasonable accuracy for the experiments considered later, we use the MRA approach and modify it slightly to suit our needs. The details of our modification are given below and for completeness, the MRA approach is discussed in \Cref{app: relate MRA}. 

\subsection{Discontinuity Detection}\label{detect shock} Recall that our FOM corresponds to a FV approximation.
With $u_i(t)$ we represent the constant value of $\solFV{t}$ inside $\mcal I_i$, where $\mcal I_i$ is the $i$-th cell defined in \eqref{partition space}. The $\dimLarge$-cells have $\dimLarge+1$ faces and we collect their indices in $\mcal E:=\{1,\dots,\dimLarge + 1\}$. With $x_e$ we represent the location of the $e$-th face, i.e. the face between $\mcal I_e$ and $\mcal I_{e+1}$. Across every face we compute the jump in $\solFV{t}$ and if the jump overshoots a given tolerance, we mark it as a potential location of discontinuity. Details are as follows. 

Let $e\in\mcal E$. With $J_e(t)$ we denote the absolute value of the jump in $\solFV{t}$ across the edge $e$ i.e.,
\begin{gather}
    J_e(t) = |u_{e}(t)-u_{e-1}(t)|,\hspB\forall e\in\mcal E. \label{def jump edge}
\end{gather}
Using $J_e(t)$, we define the set $\mcal B(t)$ that contains the indices of faces with a potential discontinuity in the adjoining cell
\begin{gather}
    \mcal B(t) := \{e\in\mcal E\hsp : \hsp J_e(t) > C\times \Delta x \}.\label{set B}
\end{gather}
Above, $C$ is user-defined and controls the number of faces that will be  contained in $\mcal B(t)$. Later, we elaborate more on the relevance of $C$. 

To compute the discontinuity location using $\mcal B(t)$, we proceed as follows. We partition $\mcal B(t)$ into sub-sets $\{\mcal B_i(t)\}_i$ such that each of $\mcal B_i(t)$ contains indices of only the adjoining faces. For instance, if $\mcal B(t) = \{1,2,4,5\}$ then $\mcal B_1(t) = \{1,2\}$ and $\mcal B_2(t) = \{4,5\}$. 
 A set $\mcal B_i(t)$ can have more than one element when, due to the numerical dissipation in the FV scheme, the discontinuity is spread out into a set of neighbouring cells, or when there are multiple discontinuities in succession. For both the cases, we compute the discontinuity location by taking the mean of all the face locations in $\mcal B_i(t)$. Equivalently,
 \begin{gather}
 z^{D}_{M,i}(t) := \frac{\sum_{e\in\mcal B_i(t)}x_e}{|\mcal B_i(t)|}\hspB\forall i\in \{1,\dots,p^D(t)\}.
 \end{gather} 
Here $z^{D}_{M,i}(t)$ denotes an approximation to the the true discontinuity location $z^D_i$, and $p^D(t)$ denotes the total number of discontinuities.

\begin{remark}
Ideally, $\mcal B(t)$ should include only those faces that have discontinuities in the adjoining cells. However, depending upon $C$'s value and the solution's behaviour away from a discontinuity, the ideal situation might not be realized. Additional faces that do not contain discontinuities in the adjoining cells might be included in $\mcal B(t)$. The inequalities given in \Cref{app: flag regions} give some indication of how the method flags different regions. We emphasize that identifying additional feature location does not ruin the calibration procedure. It only results in additional points being matched between two snapshots. However, with any additional feature it is more likely to violate the conditions (C1)-(C3), resulting in \Cref{algo: part time} generating additional reference snapshots.
\end{remark}

\subsection{Kink detection} \label{sec: kink detection}
Let $\hat \Omega(t):=\{z_i^D(t)\}_{i=1,\dots,p^D(t)}$ be a set of points where $\sol{t}$ is discontinuous. In \Cref{def: feature}, we defined kink locations as points where $D_xu(t)$ has a discontinuity in $\Omega/\hat\Omega(t)$. Thus, to find these locations, we run the discontinuity detection algorithm on $D_x\sol{t}$. To realize the algorithm we need an approximation for $D_x\sol{t}$ and $\hat{\Omega}(t)$. 

Let $D_x\solFV{t}$ be an approximation to $D_x \sol{t}$. We find $D_x\solFV{t}$ by applying central differences to $\solFV{t}$. Let $D_x u_i(t)$ be the constant value of $D_x\solFV{t}$ in the cell $\mcal I_i$. Then, $D_x u_i(t)$ is given as 
\begin{gather}
D_x u_i(t) = \frac{u_{i+1}(t) - u_{i-1}(t)}{2\Delta x}.
\end{gather}
On the continuous level, the derivative of $\sol{t}$ is a Dirac-distribution at points where $\sol{t}$ is discontinuous. However, on a spatially discrete level, the delta distribution is a collection of "spikes" in $D_x \solFV{t}$. To collect these spike we approximate every entry $z_i^D(t)$ by a ball of radius $\epsilon$ centered around $z_i^D(t)$. As an approximation to $z_i^D$ we use $x_e$, where $x_e$ is the location of the $e$-th face, $e\in\mcal B(t)$, and $\mcal B(t)$ is as given in \eqref{set B}. We set $\epsilon$ to $N^D\times\Delta x$ and we approximate $\hat\Omega(t)$ by
\begin{gather}
    \hat\Omega(t)\approx  \bigcup_{e\in \mcal B(t)}\mcal B(x_e;N^D\Delta x). \label{union balls}
\end{gather}
We choose $N^D = 3$. We use an example to motivate our choice for $N^D$. Let $\sol{t}$ be a unit-step function with a discontinuity at $z^D = x_e + l\Delta x$, where $l\in [0,1]$. It follows that 
\begin{gather*}
D_xu_{e-1}(t) = \frac{(1-l)}{2\Delta x}, \hspB D_x u_e(t) = \frac{1}{2\Delta x},\hspB D_x u_{e+1}(t) = \frac{l}{2\Delta x}.
\end{gather*}
For all the other intervals, $D_x \solFV{t} = 0$. Depending on the value of $l$, $D_x \solFV{t}$ can have a large spike in the intervals $\mcal I_{e-1}$, $\mcal I_e$ and $\mcal I_{e+1}$. Therefore, $N^D=3$ is a reasonable choice. 

\begin{remark}\label{remark: miss kinks}
With the above method, we do not detect kinks inside the union of balls given in \eqref{union balls}. However, for a small enough $\Delta x$, missing out on these kinks does not significantly increase the $m$-width of the calibrated manifold. This will be elucidated by numerical experiments.
\end{remark}

\subsection{Undetected features}\label{remark: regular uM}
Features can get smeared out by numerical dissipation and, depending upon the value of $C$ given in \eqref{set B}, might go undetected. For such cases, one can show that (at least) the semi-discrete numerical solution already has sufficient regularity to ensure a fast $m$-width decay. Let $u_i(t)$ be as defined in \Cref{detect shock} and let 
$$\frac{du_i(t)}{dt} = \frac{1}{\Delta x}\left(\mcal F(u_{i-1}(t),u_i(t))-\mcal F(u_{i}(t),u_{i+1}(t))\right)$$
be its evolution equation. Here, $\mcal F$ represents a numerical flux function, which we assume is in $W^{2,\infty}$. 

We first consider undetected discontinuities. Assume that $|u_{i\pm 1}(t)-u_i(t)|\leq C\Delta x$, in which case we do not detect a discontinuity at the face $i-1$ and $i$. 
Then, using the regularity of $\mcal F$, one can show that 
\begin{gather*}
|du_i(t)/dt|\leq 2\|\mcal F\|^2_{W^{1,\infty}}C.
\end{gather*}
In \Cref{lemma: decay estimate} we proved that $\varphi(x,\cdot)\in W^\hcoeff(\tspace)$. Motivated from this, we assume that $\varphi_M(x,\cdot)\in W^\hcoeff(\tspace)$, which is equivalent to $z_{M,j}\in W^{\hcoeff}(\tspace)$. Then, the above bound implies that, for $x\in \mcal I_i$, $u_M(\varphi_M(x,\cdot),\cdot)\in W^\hcoeff(\tspace)$. Thus, locally in $\mcal I_i$, $u_M(\varphi_M(x,\cdot),\cdot)$ has the regularity needed for a fast $m$-width decay of the calibrated manifold.

We now consider undetected kinks. Assume that $|D_x u_{i}(t)-D_x u_{i-1}(t)|\leq C\Delta x$, $|D_x u_{i+1}(t)-D_x u_{i}(t)|\leq C\Delta x$ and $|D_x u_{i+2}(t)-u_{i+1}(t)|\leq C\Delta x$, in which case we do not detect a kink at the face $i-1$, $i$ and $i+1$. Then, one can show that 
$$|d^2u_i/dt^2|\leq 4\|\mcal F\|^2_{W^{2,\infty}}(2 C^2 + C).$$
Following the same reasoning as above, the bound implies that, for $x\in \mcal I_i$, $u_M(\varphi_M(x,\cdot),\cdot)\in W^\hcoeff(\tspace)$.


\section{Numerical Experiments} \label{sec: num exp}
Let $\Xi_{m}(\snapMat)$ be as defined in \eqref{error POD}. The numerical experiments show the following two things. Firstly, with kink and discontinuity matching, $\Xi_{m}(\snapMatCalib)$ decays much faster than $\Xi_{m}(\snapMat)$. Secondly, both kink and discontinuity matching is better than only discontinuity matching. To construct numerical approximations where both kink and discontinuity detection is possible, we consider the best-approximation in $X_M$. Note that in light of the discussion in \Cref{remark: regular uM}, these numerical approximations are the ones were we expect the slowest $m$-width/singular-value decay.

Since $\Xi_m(\snapMatCalib)$ quantifies the $l^2$ error of approximating a calibrated snapshot in the span of the first $m$ left singular vectors of $\snapMatCalib$, similar to the bound in \eqref{bound approx mani}, it is possible that on increasing $m$, $\Xi_{m}(\snapMatCalib)$ stagnates at a value of $\mcal O(M^{-\frac{1}{2}})$. The following experiments will provide further elaboration. 

\begin{enumerate}
    \item \textbf{Test case-1} we consider the Burgers'  equation
    \begin{gather}
        \pd_t u + \pd_x \left(\frac{u^2}{2}\right) = 0,\hsp\text{on}\hsp\Omega\times \tspace,\hspB u(\cdot,t=0) = \mathbbm{1}_{[0,1]},\hsp\text{on} \hsp\Omega. \label{burgers}
    \end{gather}
    Above, $\mathbbm{1}_{[0,1]}$ represents a characteristic function over $[0,1]$.
    We choose $\Omega = (-0.5,3.5)$ and $\tspace = [0,4]$. On the boundary $\pd\Omega\times \tspace$, we prescribe $u=0$.
    \item \textbf{Test case-2} we consider the wave equation (rewritten as a first order system)
    \begin{gather}
        \pd_t u + A \pd_x u = 0,\hsp \text{on}\hsp\Omega\times \tspace, \label{wave equation}
    \end{gather}
    where $u = (u_1,u_2)^T$ is the solution vector and the matrix $A$ reads
    \begin{gather}
    A = \left(\begin{array}{c c}
    0 & 1 \\ 
    1 & 0
    \end{array}\right). 
    \end{gather}
    We choose $\Omega = (-0.5,3.5)$ and $\tspace = [0,2]$. As the initial data, for all $x\in\Omega$, we consider 
    \begin{gather}
   u_1(x,t=0) =  w_1(x) + w_2(x),\hspB
      u_2(x,t=0) = -w_1(x) + w_2(x),
    \end{gather}
    where $w_1(x)$ and $w_2(x)$ are two sin-function bumps given as
    \begin{equation}
    \begin{aligned}
    w_1(x) = &\frac{1}{\sqrt{2}}(\sin(\pi x) + 1)\mathbbm{1}_{[0,1]}(x),\\
     w_2(x) = &\frac{1}{\sqrt{2}}(\sin(\pi (x-2)) + 1)\mathbbm{1}_{[2,3]}(x). \label{def: sin bumps}
    \end{aligned}
    \end{equation}
   As in the previous case, on $\pd\Omega\times \tspace$, we prescribe $u=0$.
    \item \textbf{Test case-3} we consider the Sod's shock tube problem that involves the Euler's equation given as
    \begin{gather}
    \pd_t \left(\begin{array}{c}
    \rho\\
    \rho v\\
    E
    \end{array}\right) + \pd_x \left(\begin{array}{c}
    \rho v\\
    \rho v^2 + P\\
    E v + P v
    \end{array}\right) =0,\hsp\text{on}\hsp\Omega\times\tspace. \label{Euler equations}
    \end{gather}
    Above, $\rho$, $v$, $P$ and $E$ represent the density, the velocity, the pressure and the total energy, respectively. For an ideal gas, $P = (\gamma-1)\rho e$, where $\gamma$ represent the gas constant and $e$ is the internal energy related to the total energy via $\rho e = E-\rho v^2/2$. We consider a mono-atomic ideal gas for which $\gamma = 5/3$. We choose $\Omega = (-0.5,0.5)$ and $\tspace=[0,0.2]$. As the initial data, we consider a fluid at rest with the density and the pressure given as
    \begin{gather}
    \rho(x,t=0)=\begin{cases}
    1,\hspB &x\leq 0\\
    0.125,\hspB &x>0
    \end{cases},\hspB     P(x,t=0)=\begin{cases}
    1,\hspB &x\leq 0\\
    0.1,\hspB & x>0
    \end{cases}.
    \end{gather}
 The waves emanating from the initial discontinuity do not reach the boundary therefore, we take the boundary data from the initial values. 
    \item \textbf{Test case-4} we consider the linear advection equation with time-dependent boundary data
    \begin{equation}
    \begin{aligned}
    \pd_t u(x,t) + \beta \pd _x u(x,t) = 0,\hspB&\forall (x,t)\in\Omega\times \tspace,\\
     u(x,t=0) = (\sin(\pi x)+1)\mathbbm{1}_{[0,1]}(x)\hspB&\forall x\in\Omega,\\
    u(x=0,t) = \mathbbm{1}_{[0.1,0.5]}(t),\hspB&\forall t\in\tspace. 
    \end{aligned} \label{advection}
    \end{equation}
    We set $\beta = 1$, $\Omega = (-0.5,3.5)$ and $\tspace = [0,1]$. 
\end{enumerate}
For all the test cases, we partition $\Omega$ into $M = 2\times 10^3$ elements, and consider $10^3$ uniformly placed time instances inside $\tspace$. We choose $\mcal K_1 = 5$, $\mcal K_2 = 3$ and $C = 50$. For all the test cases, we project the exact solution onto the FV space. Details of the exact solution are given later. We compute all the $L^2(\Omega)$ inner-products with $10$ Gauss-Legendre quadrature points in each cell.

\subsection{Test case-1} The unique entropy solution to the problem in \eqref{burgers} reads
\begin{equation}
\begin{aligned}
u(x,t) := &\begin{cases}
\frac{x}{t},\hsp & x\in [0,t)\\
1,\hsp & x\in [t,1 + \frac{t}{2})\\
0,\hsp & \text{else}
\end{cases},\hspB\forall t\in [0,2),\\
 u(x,t) := &\begin{cases}
\frac{x}{t},\hsp & x\in [0,\sqrt{2 t})\\
0,\hsp & \text{else}
\end{cases},\hspB \forall t\in [2,4].
\end{aligned}
\end{equation}
The exact solution has two discontinuities at $t=0$. One of the discontinuities gives rise to two kinks (a rarefaction fan), the other remains as a discontinuity. At $t=2$, one of the kinks collides with a discontinuity to form a single discontinuity. Around $t=0$, the two kinks are very close to each other and are identified as a single discontinuity in the numerical solution; see \Cref{test-1 feature loc}. As time progresses, the two kinks move away from each other and are identified correctly.

Let $\mcal E(\Delta x)$ represent the maximum of the error in feature location for a grid size $\Delta x$ i.e.,
\begin{gather}
\mcal E(\Delta x) := \max_j\|z_{M,j}-z_j\|_{L^{\infty}(\tspace)}.
\end{gather}
Recall that $\Delta x = |\xMax-\xMin|/M$. \Cref{test-1 error feature} shows $\mcal E(\Delta x)$ for different grid sizes. We vary the number of spatial elements $M$ from $5\times 10^2$ to $3\times 10^3$ in steps of $2\times 10^2$. We choose the threshold $C$ in the discontinuity location identification such that $C/M$ remains constant at $2.5 \cdot 10^{-2}$.  We make the following two observations. Firstly, although not monotonically, $\mcal E(\Delta x)$ decreases with $\Delta x$. Secondly, $\mcal E(\Delta x)$ stays close to $\Delta x$ and can get smaller than $\Delta x$ as $\Delta x$ decreases. Thus, at least for the current feature location identification procedure and for the current test case, the assumption on the error in feature location made in \Cref{coro: width decay} is justified.

The dashed lines in \Cref{test-1 feature loc} show the temporal locations of the reference snapshots resulting from \Cref{algo: part time}. The algorithm provides $N = 5$ (with $N$ as given in \Cref{time partitions}) different reference snapshots located at $t=0$, $t = 0.02$, $t = 1.60$, $t=1.92$ and $t = 1.98$, respectively. The first reference snapshot is the initial data that is matched to a few subsequent snapshots, which is a result of identifying the two close-by kinks as a single discontinuity. The second reference snapshot is at a time instance when our feature identifier can distinguish between the two kinks. The third and the fourth reference snapshot is selected because the features come too close to each other, violating either the condition (C3) given in \eqref{assume naive} or the lower-bound on the minimum feature distance given in \eqref{lower bound dist}. The last reference snapshot is selected after the kink collides with the discontinuity, it matches to all the subsequent snapshots. Note that in the exact solution, the kink collides with the discontinuity at $t=2$. However, numerically, as mentioned in \Cref{remark: miss kinks}, we miss out on kinks that lie very close to a discontinuity therefore, already at $t=1.98$ we detect only the discontinuity and not the kink that interacts with it. 

\Cref{test-1 E} compares $\Xi_{m}(\snapMat_i)$ to $\Xi_{m}(\snapMatCalibSub{i})$ and shows that, for all values of $i$ and $m$, $\Xi_{m}(\snapMatCalibSub{i})$ is smaller than $\Xi_{m}(\snapMat_{i})$. Since $\snapMat_1$ contains only four snapshots, the value of $\Xi_{m}(\snapMatCalibSub{1})$ does not significantly differ from $\Xi_{m}(\snapMat_{1})$. For all the other sub-matrices, the value of $\Xi_{m}(\snapMatCalibSub{i})$, already for $m = 1$, is at least $10^{-4}$ times smaller than $\Xi_{m}(\snapMat_{i})$. Let us emphasize that $m=1$ is just $0.05\%$ of $M$ (the dimensionality of the FOM).

 For $i=4,5$, as $m$ is increased, $\Xi_{m}(\snapMatCalibSub{i})$ stagnates. Varying the value of $M$ from $10^3$ to $3\times 10^3$ in steps of $2\times 10^2$ showed that the stagnation value is $\mcal O(M^{-0.8})$, which is $\mcal O(M^{-0.3})$ times better than (the $M$-dependent part of) the bound on the $m$-width developed in \eqref{width decay rate}. A possible reason for this stagnation could be the error in feature location.
 
 For $i=2$, the matrix $\snapMat_i$ contains snapshots that are either rarefaction fans or constants between any two features, thus satisfying the condition in \eqref{cond trivial}.
 This results in the calibrated manifold consisting of a single function. Ideally, the calibrated snapshot matrix should have a rank close to one and for $m=1$, $\Xi_{m}(\snapMatCalibSub{i })$ should be (very) close to zero. However, as \Cref{test-1 E} depicts, because of the error in feature location, this ideal situation is not realized in practice and the value $\Xi_{m}(\snapMatCalibSub{i})$ is far away from zero. Nevertheless, for $m=13$, $\Xi_{m}(\snapMatCalibSub{i})$ reaches (machine precision) zero. We attribute this convergence to the fact that the error in identifying a feature location is $\mcal O(M^{-1})$ and that the calibrated manifold $\calibMani{.}{\tspace_i}$ consists of a single function. Observance of a similar behaviour in other experiments corroborates our claim.
 
 \subsubsection{Discontinuity matching} We repeat the above experiment but with only discontinuity matching. With $\snapMatCalib^D$ we represent the resulting calibrated snapshot matrix. \Cref{algo: part time} generates two reference snapshots i.e., $N=2$. The temporal location of these two reference snapshots are shown in \Cref{shock loc}. Both the reference snapshots are close to $t=0$. The first reference is the initial data and is matched to a few subsequent snapshots. The second reference snapshot is at a time-instance when we can uniquely identify the two kinks, leaving us with a single discontinuity.
 
 \Cref{shock kink comp} compares $\Xi_m(\snapMatCalibSub{i})$ to $\Xi_m(\snapMatCalibSub{i}^D)$. For $i=1$, both $\Xi_m(\snapMatCalibSub{i})$ and $\Xi_m(\snapMatCalibSub{i}^D)$ have the same values. This is as expected, since the two close-by kinks are identified as a discontinuity. For $i >1$ and for all $m\in [1,20]$, $\Xi_m(\snapMatCalibSub{i})$ is at least two orders of magnitude smaller than $\Xi_m(\snapMatCalibSub{i}^D)$. The difference is more prominent for smaller values of $m$. Already for $m=1$, $\Xi_m(\snapMatCalibSub{i})$ is four order of magnitude smaller than $\Xi_m(\snapMatCalibSub{i}^D)$. The experiment clearly establishes the benefit of including both kinks and discontinuities in the feature set.

\begin{figure}[ht!]
\centering
\subfloat []{
\includegraphics[width=2.4in]{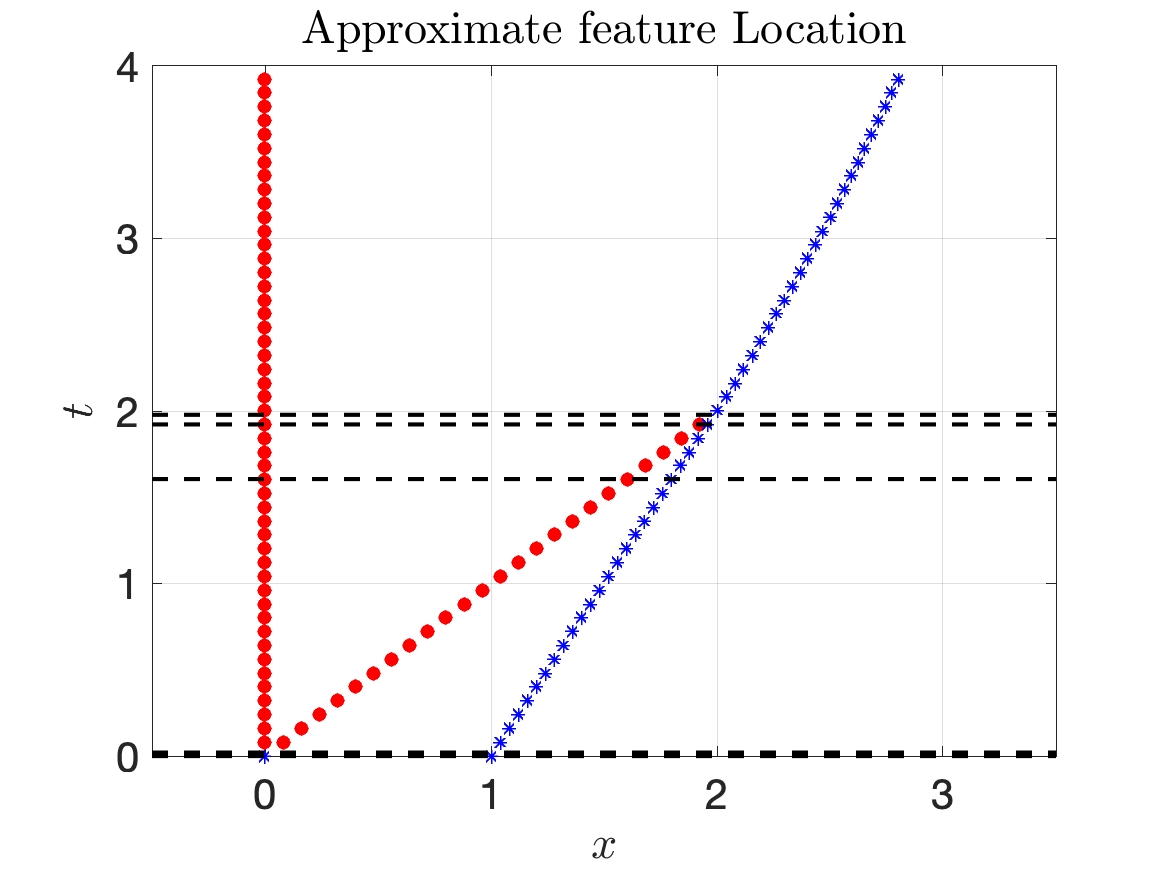} 
\label{test-1 feature loc}}
\subfloat []{
\includegraphics[width=2.4in]{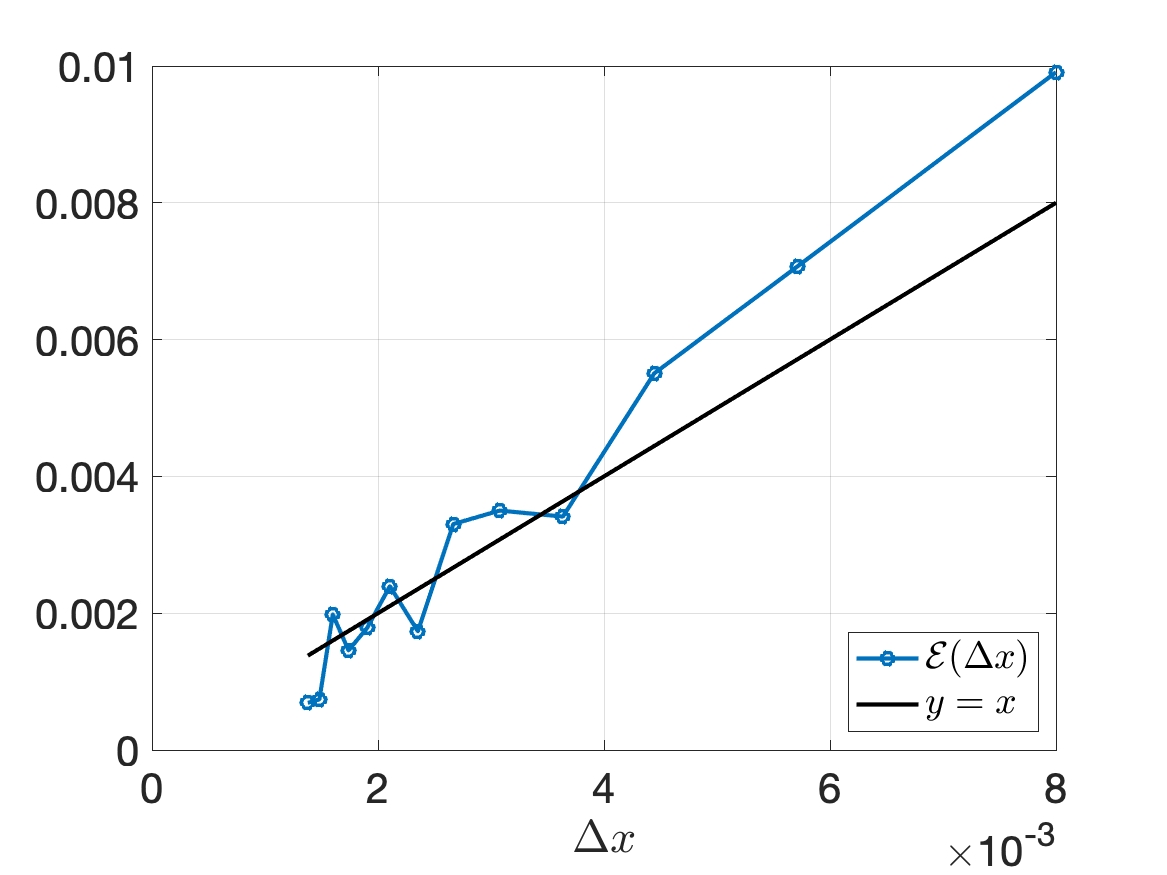}
\label{test-1 error feature}}
\hfill
\subfloat []{
\includegraphics[width=2.4in]{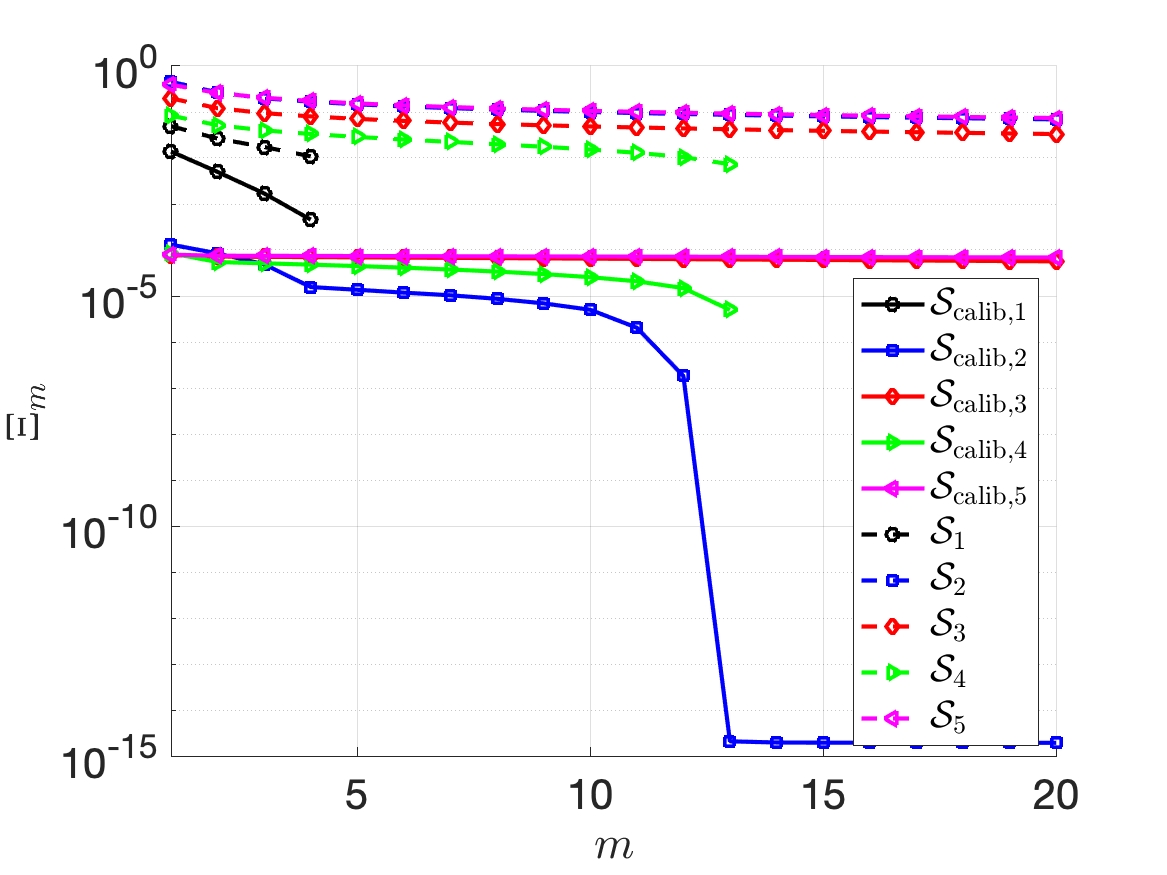}
\label{test-1 E}}
\caption{\textit{Results for test case-1. Both kinks and discontinuities included in the feature set. (a) Time-trajectory of the different features. Kink and discontinuity locations shown in red and blue, respectively. The dashed black lines show the temporal locations of the reference snapshots. (b) Error in feature location for different $\Delta x$. (c) Comparison of $\Xi_m(\snapMat_{i})$ to $\Xi_m(\snapMatCalibSub{i})$. The y-axis of (c) is on a log-scale.}}	
\end{figure} 

\begin{figure}[ht!]
\centering
\subfloat []{
\includegraphics[width=2.4in]{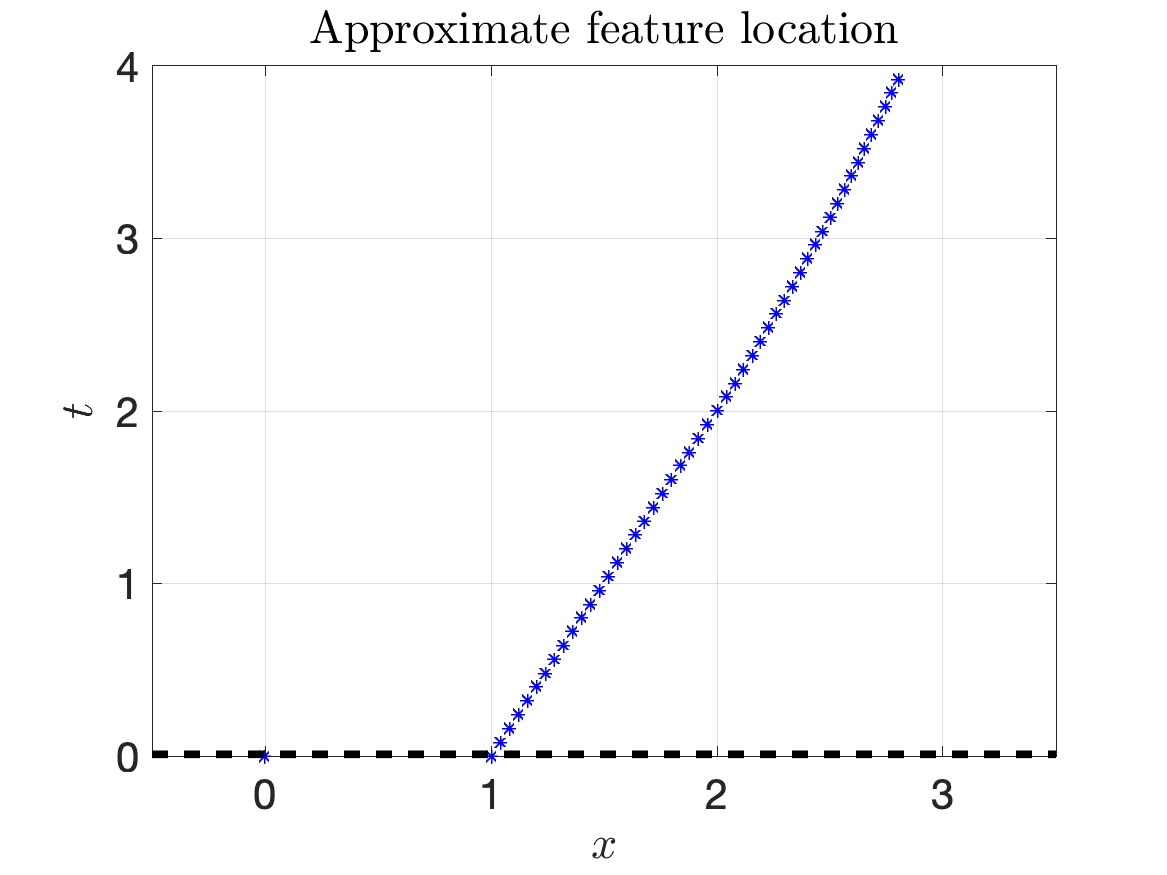} 
\label{shock loc}
}
\hfill
\subfloat []{
\includegraphics[width=2.4in]{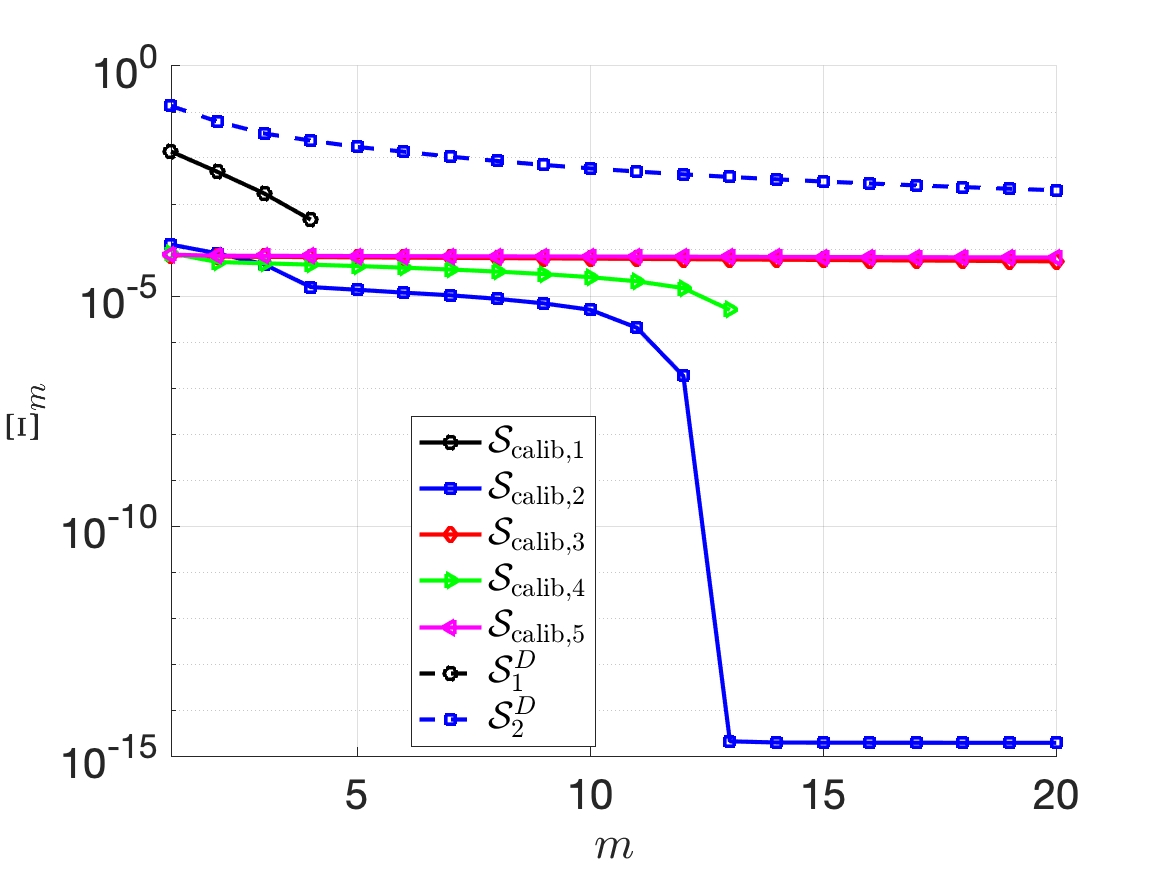}
\label{shock kink comp}}
\hfill
\caption{\textit{Results for test case-1. Only discontinuities included in the feature set. (a) Time-trajectory of the different features. Kink and discontinuity locations shown in red and blue, respectively. The dashed black lines show the temporal locations of the reference snapshots. (c) Comparison of $\Xi_m(\snapMatCalibSub{i})$ to $\Xi_m(\snapMatCalibSub{i}^D)$. The y-axis of (b) is on a log-scale.}}	
\end{figure}

\subsection{Test case-2}
With the help of the Riemann invariants, for all $(x,t)\in\Omega\times\tspace$, one can conclude that the exact solution to the wave equation \eqref{wave equation} is given as 
    \begin{gather}
   u_1(x,t) = w_1(x-t) + w_2(x+t),\hspB  u_2(x,t) = -w_1(x-t) + w_2(x+t).
    \end{gather}
The functions $w_1$ and  $w_2$ are as given in \eqref{def: sin bumps}. 
Both $u_1$ and $u_2$ contain two discontinuities, which interact at four different time instances. For $u_1$, the time-trajectory of the different discontinuities is shown in \Cref{test-2 feature loc}. The algorithm accurately identifies the four discontinuities. 

We discuss the results for $u_1$, similar results were observed for $u_2$. \Cref{algo: part time} generates $N=18$ different reference snapshots. The temporal locations of these snapshots are shown in \Cref{test-2 feature loc}. Similar to the previous test case, the reference snapshot changes frequently when features come close, or interact, with each other. To study $\Xi_m$, for the simplicity of exposition, out of the 18 different subsets $\{[t]_i\}_{i=1,\dots,18}$, we select the first four with the largest number of snapshots. These four subsets lie inside $(0,0.5)$, $(0.5,1)$, $(1,1.5)$ and $(1.5,2)$, respectively, which are also the time-intervals with no feature interaction. 

For these four subsets, \Cref{test-2 E14} and \Cref{test-2 E23} compare $\Xi_m(\snapMat_i)$ to $\Xi_m(\snapMatCalibSub{i})$. Already for $m=1$, the value of $\Xi_m(\snapMatCalibSub{1/18})$ is $\approx 10^{-5}$ and is machine-precision zero for $m=3$. For the same value of $m$, the value of $\Xi_m(\snapMat_{1/18})$ is $\approx 1$. The value of $\Xi_m(\snapMatCalibSub{7/12})$ behaves differently. For $m=4$ and larger, it does not appear to converge to zero and stagnates at $\approx 10^{-4}$. For the same value of $m = 4$, the value of $\Xi_m(\snapMat_{7/12})$ is $\approx 10^{-1}$. This is $10^3$ times larger than the value of $\Xi_m(\snapMatCalibSub{7/12})$. 

Note that $\snapMat_{1/18}$ contains snapshots that have two sin-bumps that do not interact with each other and have a constant speed of one. One can conclude that this results in the calibrated manifold $\calibMani{}{\tspace_{1/18}}$ consisting of a single function. \Cref{test-2 S14} shows the snapshots in $\snapMatCalibSub{1}$. The snapshots change (very) little over time, with no change being visible. In contrast, as depicted by \Cref{test-2 S23}, the snapshots in $\snapMatCalibSub{7}$ change substantially over time. This could explain the superior calibration of $\snapMat_{1/18}$ as compared to $\snapMat_{7/12}$.

\begin{figure}[ht!]
\centering
\subfloat []{
\includegraphics[width=2.4in]{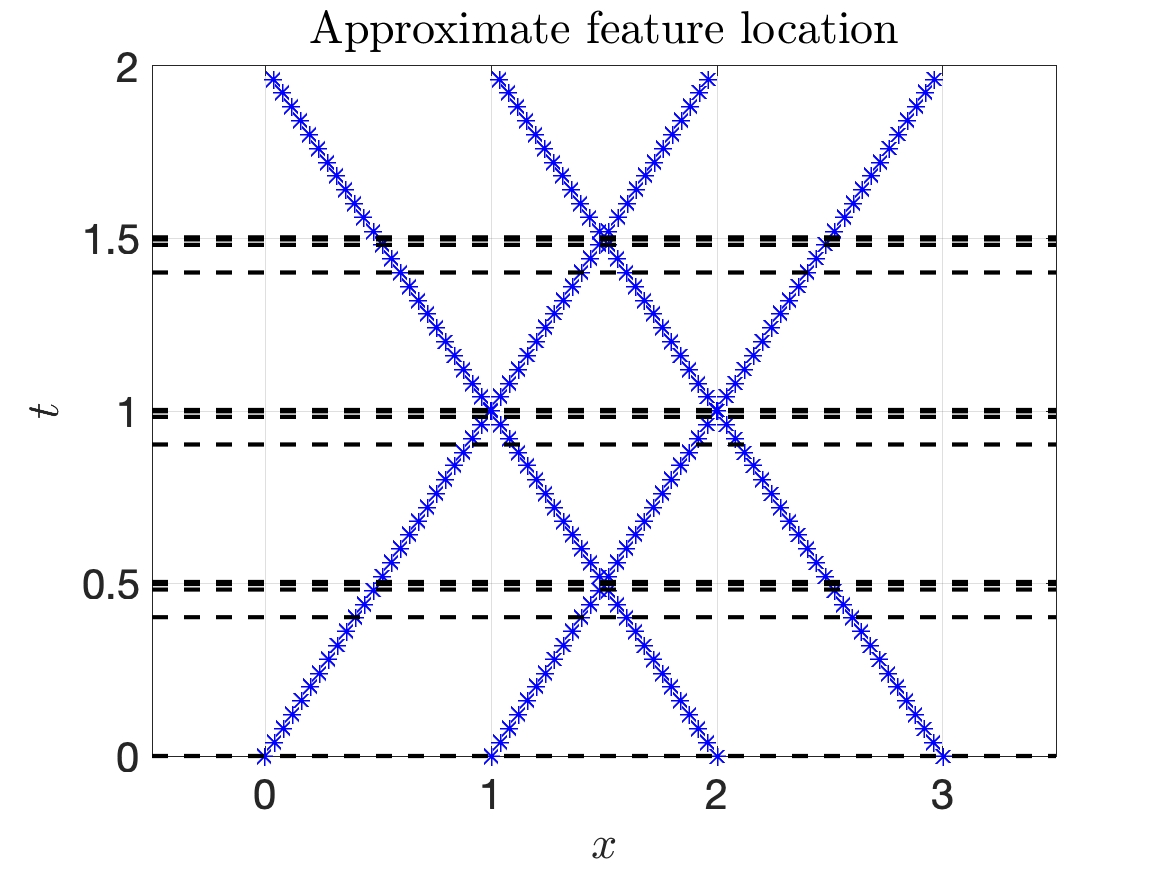} 
\label{test-2 feature loc}
}
\hfill
\subfloat []{
\includegraphics[width=2.4in]{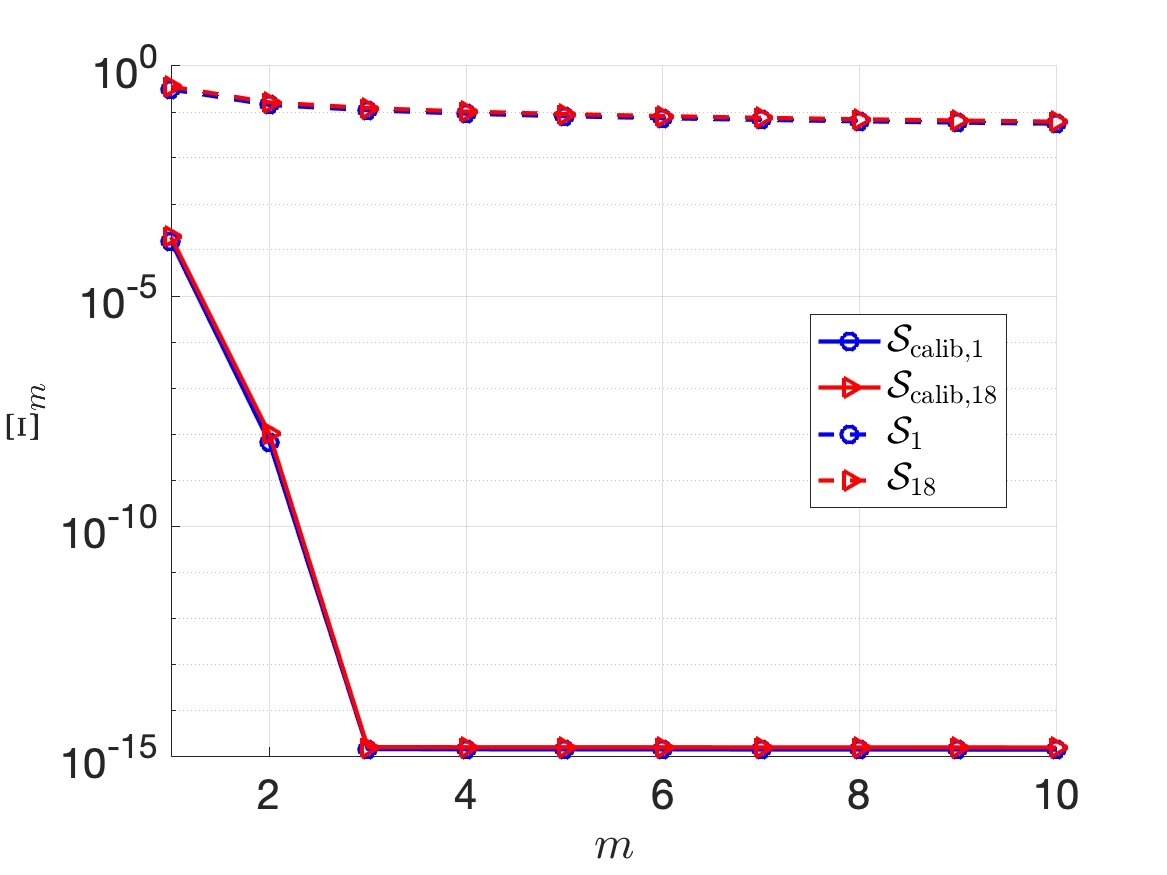} 
\label{test-2 E14}
}
\hfill
\subfloat []{
\includegraphics[width=2.4in]{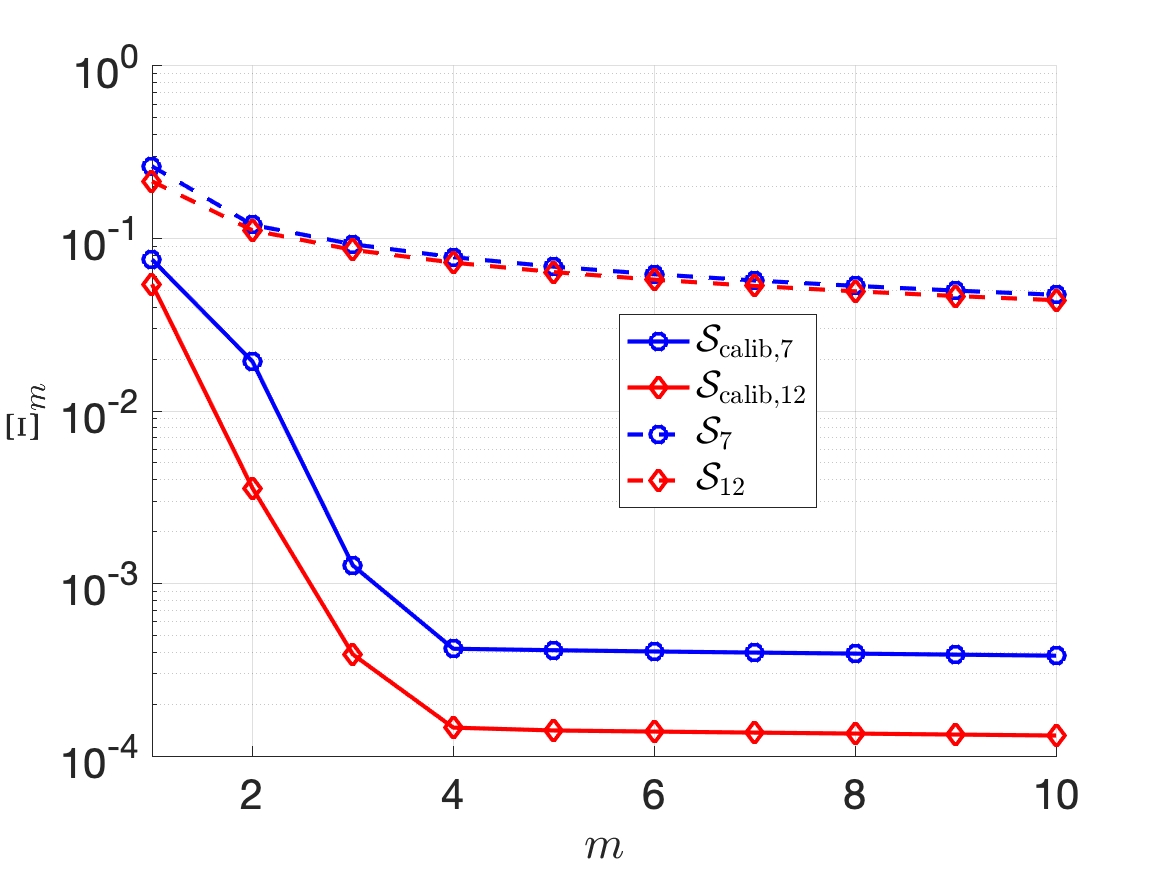} 
\label{test-2 E23}
}
\caption{\textit{Results for test case-2. (a) Time-trajectory of the approximate feature locations. Kink and discontinuity locations shown in red and blue, respectively. The dashed black lines show the temporal locations of the reference snapshots. (b) Compares $\Xi_m(\snapMat_{1/18})$ to $\Xi_m(\snapMatCalibSub{1/18})$. (c) Compares $\Xi_m(\snapMat_{7/12})$ to $\Xi_m(\snapMatCalibSub{7/12})$. The y-axis of (b) and (c) is on a log-scale.}}	
\end{figure} 

\begin{figure}[ht!]
\centering
\subfloat []{
\includegraphics[width=2.4in]{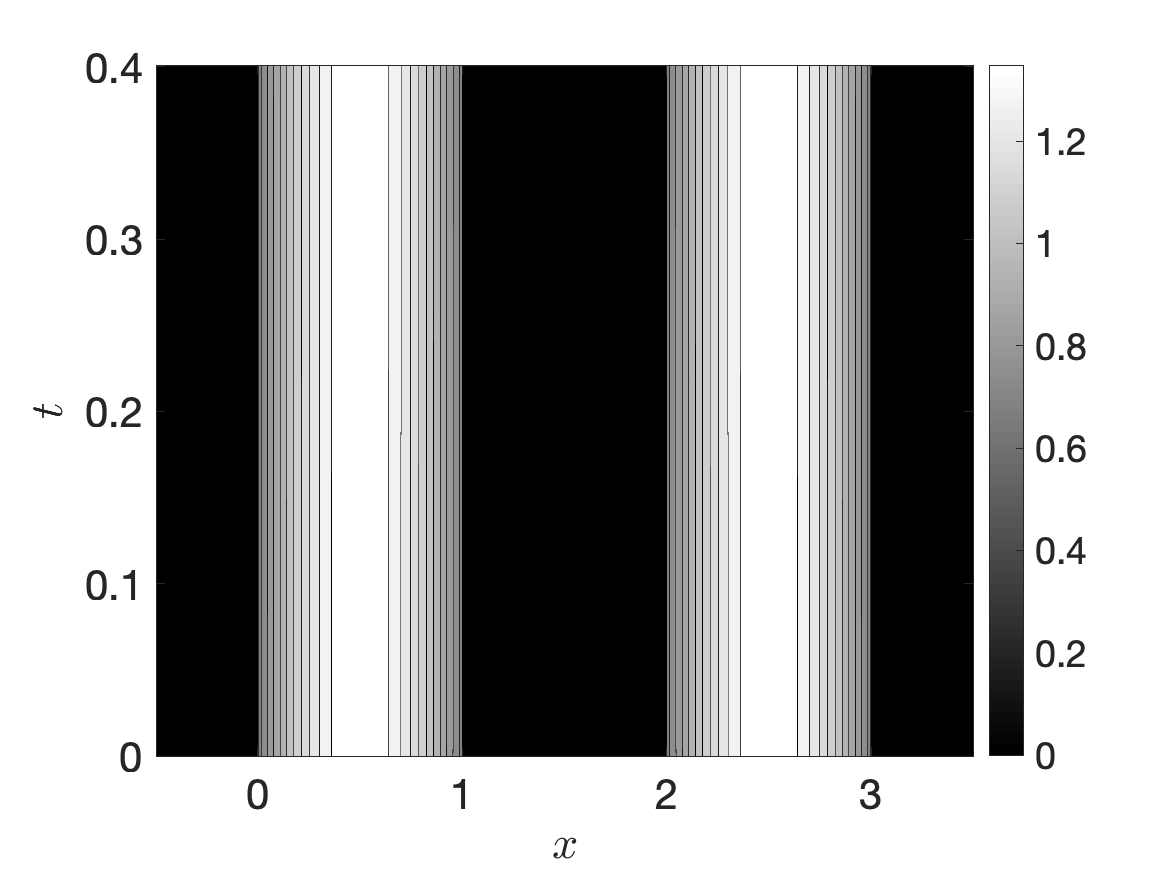} 
\label{test-2 S14}
}
\hfill
\subfloat []{
\includegraphics[width=2.4in]{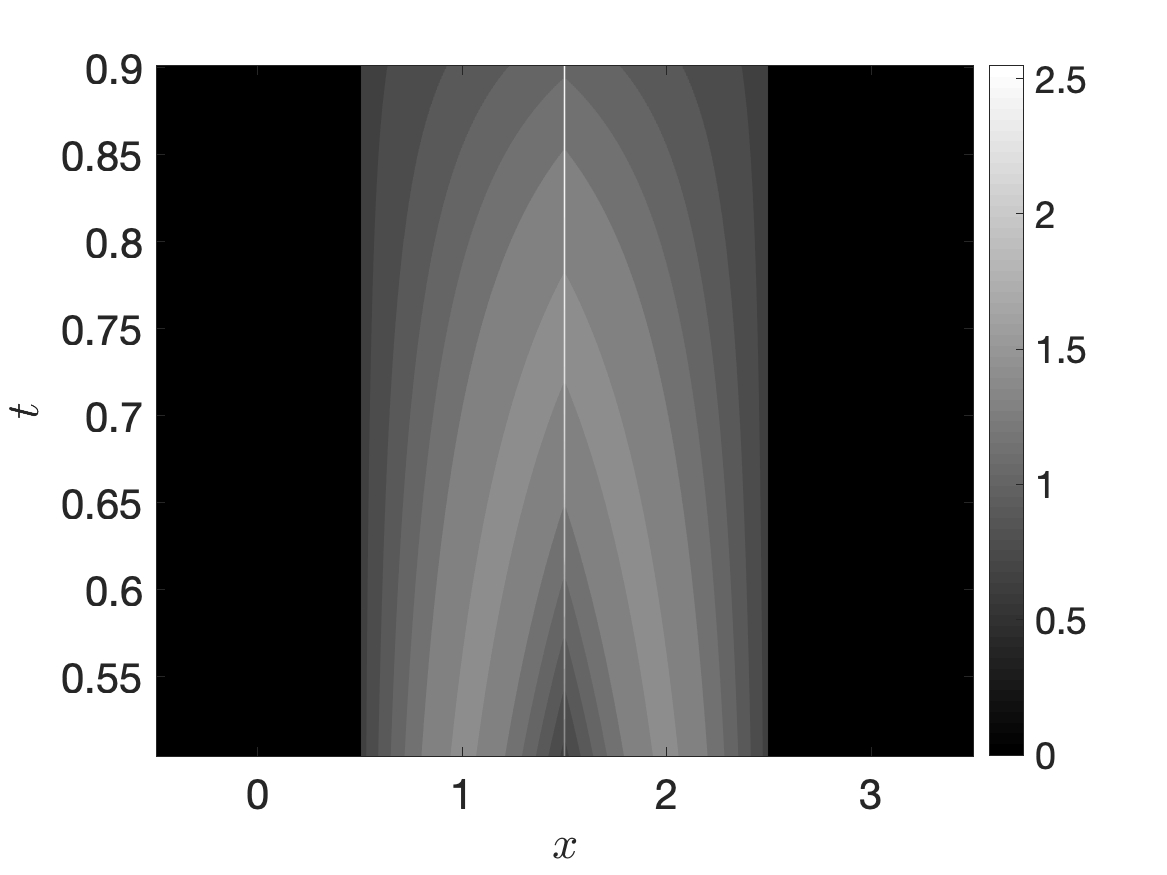} 
\label{test-2 S23}
}
\hfill
\caption{\textit{Results for test case-2. (a) and (b) show the snapshots in $\snapMatCalibSub{1}$ and $\snapMatCalibSub{7}$, respectively.}}	
\end{figure}

\subsection{Test case-3}
An exact solution to the Sod's shock tube problem can be found in \cite{ToroBook}. For brevity, we do not repeat the exact solution here. We present the results for velocity ($v$) and density ($\rho$). The results for pressure ($P$) are similar to that for density ($\rho$) and are not discussed for brevity. 

\subsubsection{Results for density ($\rho$)}
The initial data has a single discontinuity that splits into a rarefaction fan with two kinks and two discontinuities; see \Cref{test-3 exact feature rho}. The approximate feature trajectories are shown in \Cref{test-3 feature rho}. Around $t=0$, the kinks are too close to each other and are identified as a single discontinuity. For $t\in (0.01,0.1)$, because of a large slope inside the rarefaction fan, the algorithm is unable to distinguish between the two kinks and identifies the midpoint of the two kinks as the kink location.  Only after $t=0.1$, the spread of the rarefaction fan allows for an accurate identification of the two kinks.

\Cref{algo: part time} generates $N=11$ different reference snapshots, the location of which are shown in \Cref{test-3 feature rho}. Because the features are too close to each other, the reference snapshot changes frequently close to $t=0$. Around $t=0.1$, the two kinks are identified correctly and the algorithm generates an additional reference snapshot.

To study $\Xi_m$, out of $\{[t]_i\}_{i=1,\dots,N}$, we select the two largest subsets. These two subsets lie inside $(0.02,0.1)$ and $(0.1,0.2)$, respectively. \Cref{test-3 E rho} compares $\Xi_m(\snapMat_{8/9})$ to $\Xi_m(\snapMatCalibSub{8/9})$. For both $i =8$ and $i=9$, $\Xi_m(\snapMatCalibSub{i})$ decays much faster than $\Xi_m(\snapMat_{i})$. For $m=10$, which is $0.5\%$ of $M$, calibration provides at least one order-of-magnitude improvement, with the results for $i=9$ being better than those for $i=8$. Precisely, 
\begin{gather}
\Xi_{10}(\snapMatCalibSub{8})\approx 5\times 10^{-2} \times \Xi_{10}(\snapMat_{8}),\hspB \Xi_{10}(\snapMatCalibSub{9})\approx  1\times 10^{-2} \times\Xi_{10}(\snapMat_{9}).
\end{gather}
As $m$ increases, the difference between $\Xi_m(\snapMatCalibSub{i})$ and $\Xi_m(\snapMat_{i})$ becomes larger. For $m=50$, which is $2.5\%$ of $M$, we find an improvement of at least two orders of magnitude
\begin{gather}
\Xi_{50}(\snapMatCalibSub{8})\approx 10^{-2} \times\Xi_{50}(\snapMat_{8}),\hspB \Xi_{50}(\snapMatCalibSub{9})\approx 7\times 10^{-3} \times\Xi_{50}(\snapMat_{9}).
\end{gather}

\begin{figure}[ht!]
\centering
\subfloat []{
\includegraphics[width=2.4in]{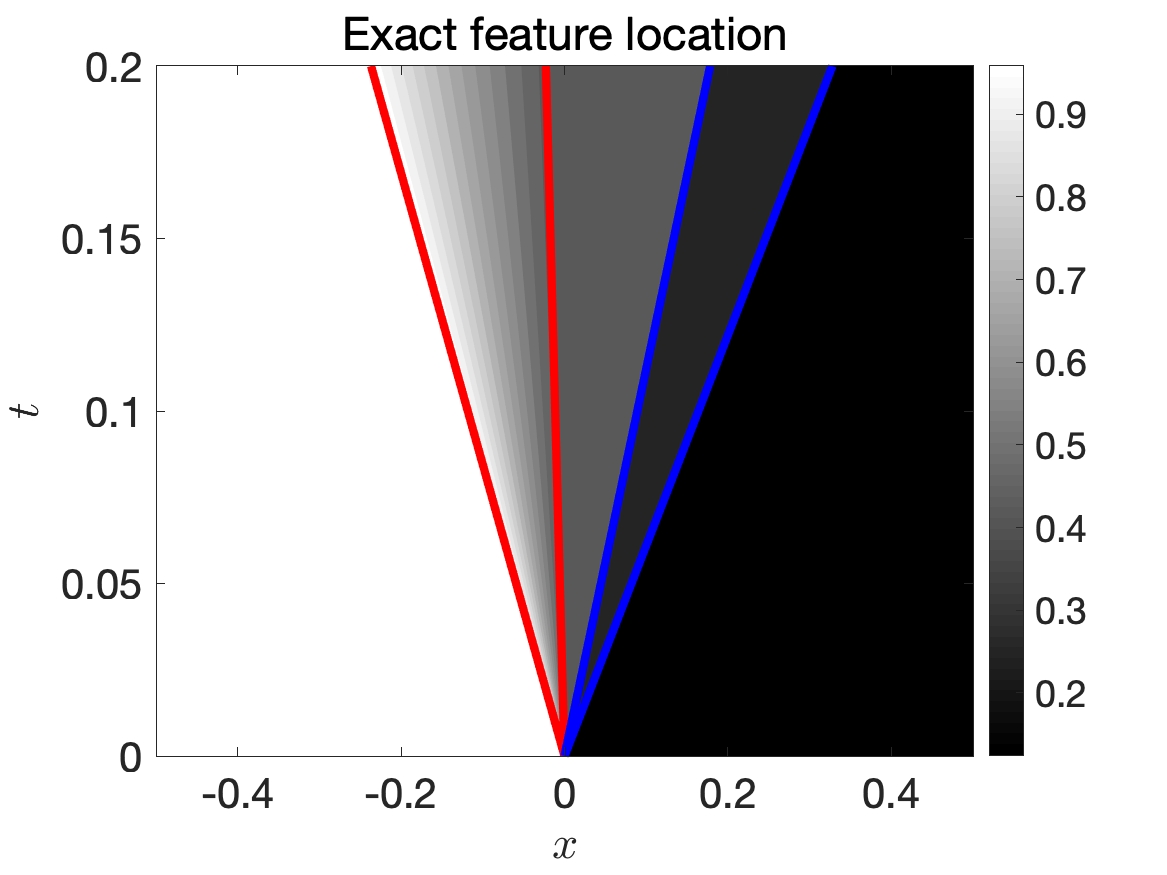} 
\label{test-3 exact feature rho}
}
\hfill
\subfloat []{
\includegraphics[width=2.4in]{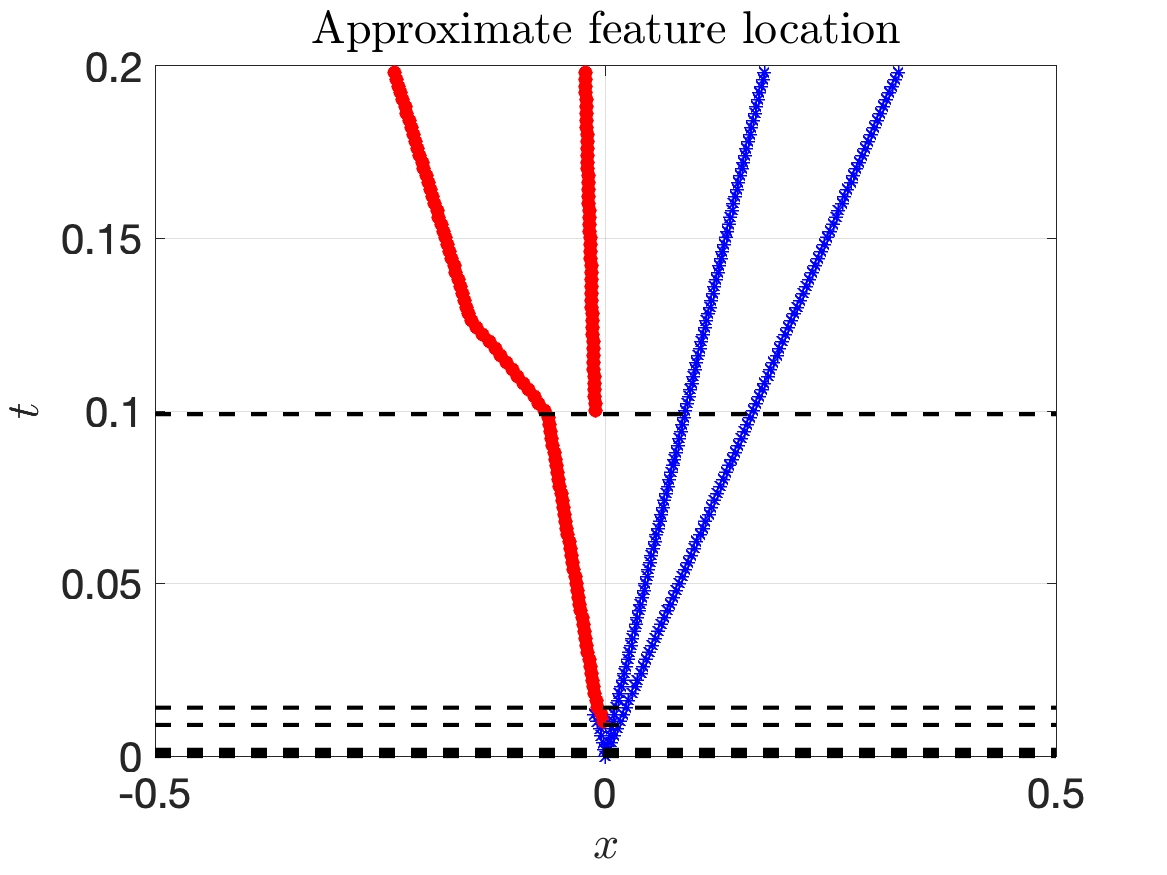} 
\label{test-3 feature rho}
}
\hfill
\subfloat []{
\includegraphics[width=2.4in]{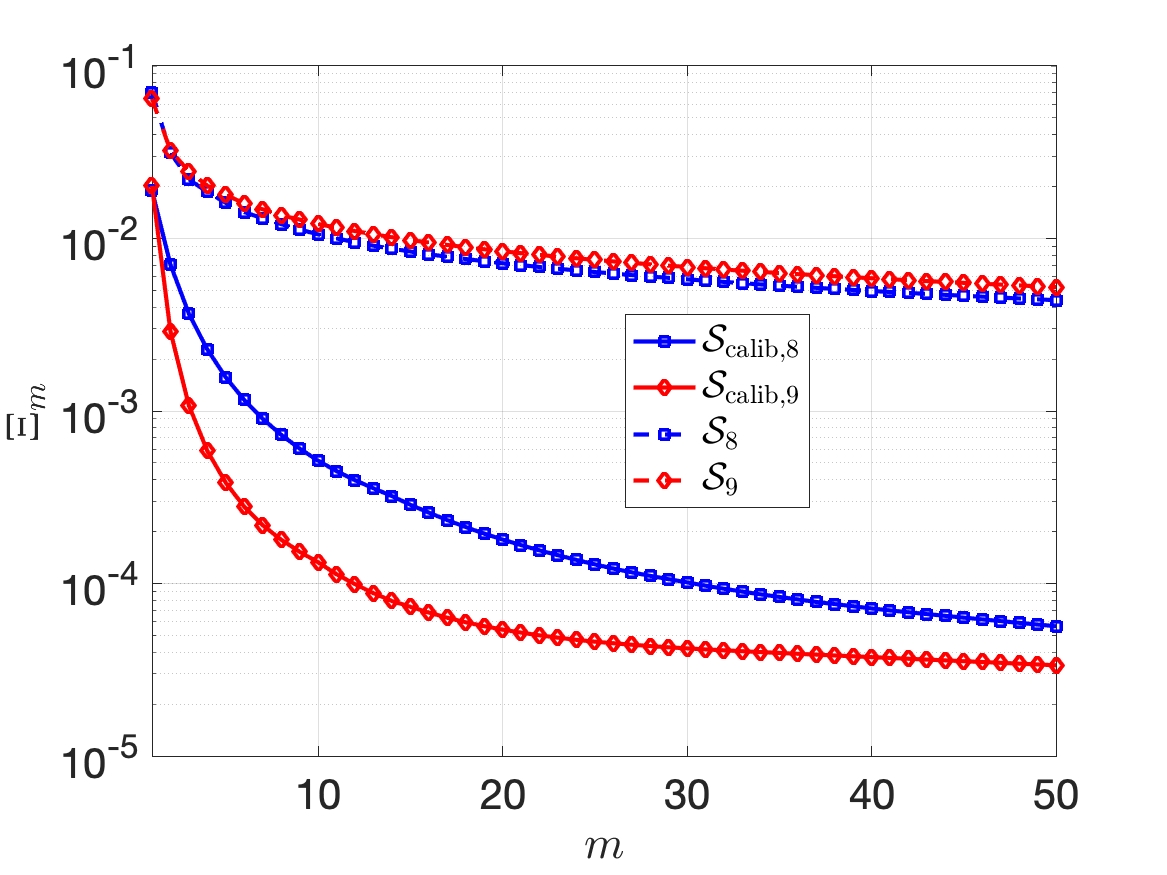} 
\label{test-3 E rho}
}
\hfill
\caption{\textit{Test case-3: results for $\rho$. (a) and (b) show the exact and the approximate feature trajectory, respectively. Kinks are shown in red and the discontinuities in blue. Dashed lines in (b) show the temporal locations of the reference solutions. (c) Compares $\Xi_m(\snapMat_{8/9})$ to $\Xi_m(\snapMatCalibSub{8/9})$. The y-axis in (c) is on a log-scale. }}
\end{figure}

\subsubsection{Results for velocity ($v$)}
Apart from $t=0$, $v(\cdot,t)$ has two kinks and a discontinuity. Similar to test case-1, the two kinks are identified once they have moved sufficiently far away from each other, otherwise they are identified as a single discontinuity. The discontinuity is identified accurately at all time instances; see \Cref{test-3 feature loc u}. 

\Cref{algo: part time} generates $N=5$ different reference snapshots. Most of these reference snapshots are close to $t=0$. The time interval $(0.01,0.2)$ is the largest subset of $\tspace$ where the reference snapshot does not change. For this time-interval, in \Cref{test-3 E u} we compare $\Xi_m(\snapMat_{i})$ to $\Xi_m(\snapMatCalibSub{i})$. Already for $m=1$, we find that  $\Xi_m(\snapMatCalibSub{5})\approx 10^{-3}$, which is two orders of magnitude smaller than $\Xi_m(\snapMat_{5})$.
For $m=30$, which is $1.5\%$ of $M$, $\Xi_{30}(\snapMatCalibSub{5})$ is (machine precision) zero, whereas $\Xi_{30}(\snapMat_{5})$ is $6.4\times 10^{-2}$.

\begin{figure}[ht!]
\centering
\subfloat []{
\includegraphics[width=2.4in]{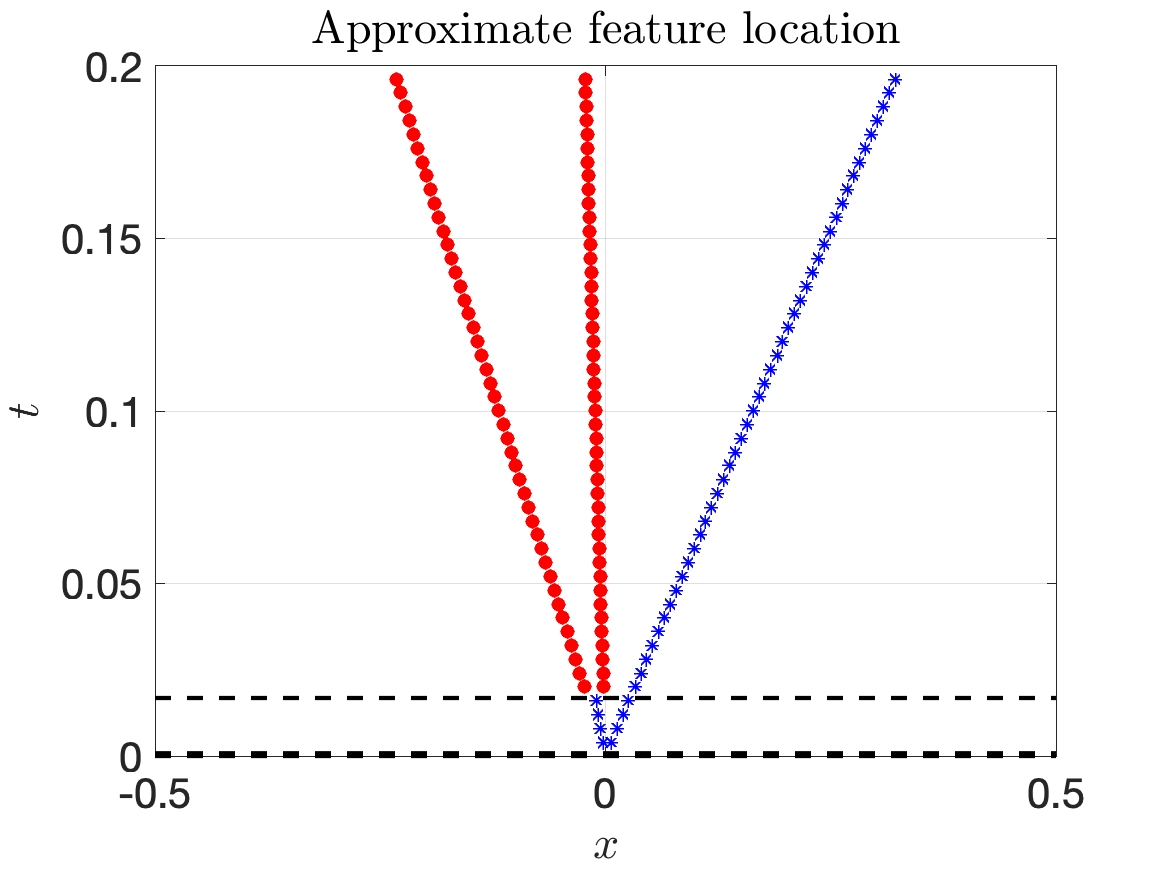} 
\label{test-3 feature loc u}
}
\hfill
\subfloat []{
\includegraphics[width=2.4in]{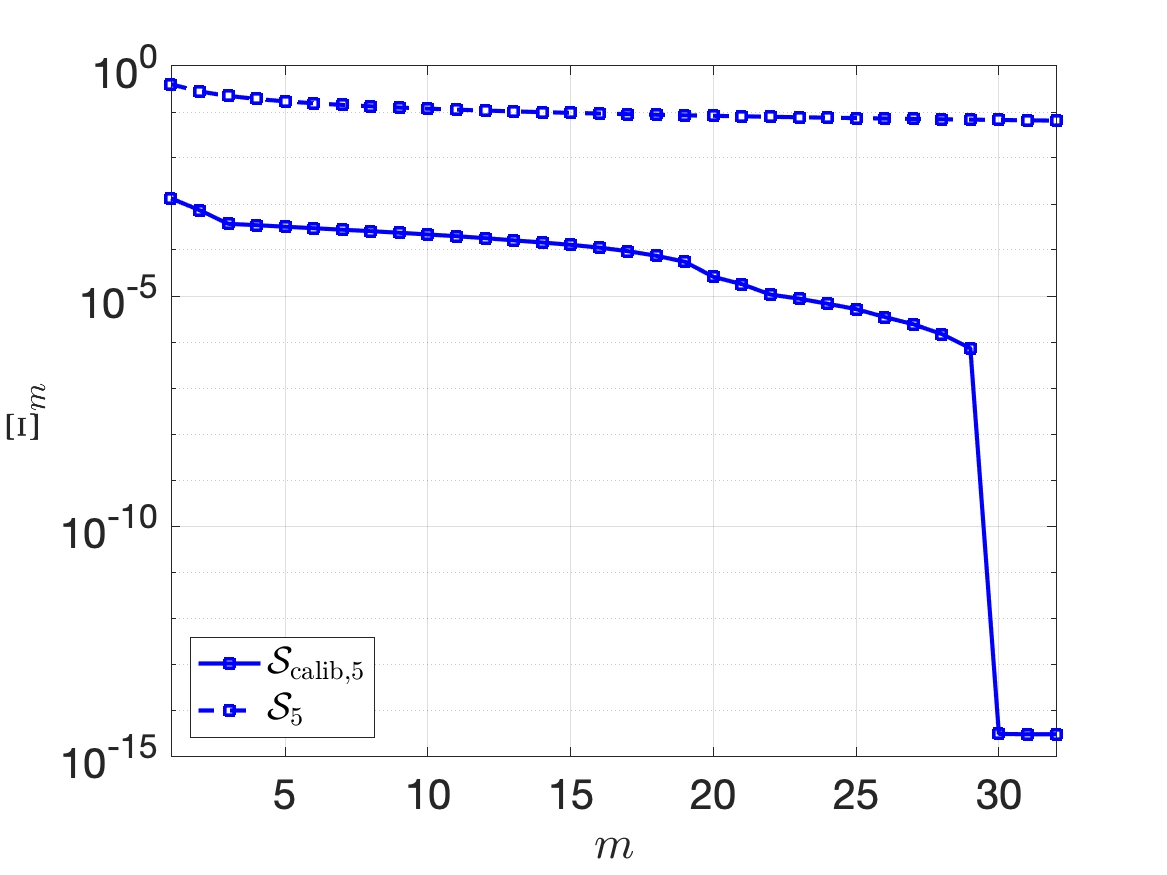} 
\label{test-3 E u}
}
\hfill
\caption{\textit{Test case-3: results for the velocity $v$. (a) Approximate feature location. (b) Compares $\Xi_m(\snapMat_5)$ to $\Xi_m(\snapMatCalibSub{5})$. The y-axis in (b) is on a log-scale. }}	
\end{figure}

\subsection{Test case-4}
An exact solution to \eqref{advection} is given as 
\begin{equation}
\begin{aligned}
 u(x,t) = &\mathbbm{1}_{[0.1,0.5]}\left(t-\frac{x-\xMin}{\beta}\right),\hspB\forall x\in (0,\xMin+\beta t],t\in\tspace,\\
u(x,t) = &(\sin(\pi(x-\beta t)) + 1)\mathbbm{1}_{[0,1]}(x-\beta t),\hspB\forall x\in (\xMin + \beta t,\xMax),t\in\tspace.\\
\end{aligned}
\end{equation}
For $t\in [0,0.1)$, the solution contains two discontinuities that move to the right. At $t=0.1$ and $t=0.5$, two additional discontinuities enter from the left boundary. \Cref{test-4 feature loc} shows the approximate location of these discontinuities. \Cref{algo: part time} generates $N=11$ different reference snapshots. The reference snapshot changes when a new discontinuity enters from the boundary. 

\Cref{test-4 E} compares $\Xi_m(\snapMat_{i})$ to $\Xi_m(\snapMatCalibSub{i})$ for the three largest subsets $[t]_i$. Clearly, $\Xi_m(\snapMatCalibSub{i})$ decays much faster than $\Xi_m(\snapMat_{i})$, and is zero for $m=3$. For the same value of $m$, $\Xi_m(\snapMat_{i})$ is $\approx 2\times 10^{-2}$. With the above exact solution, it is easy to check that the calibrated manifold consists of a single function, which could explain the great improvement offered by calibration.  

\begin{figure}[ht!]
\centering
\subfloat []{
\includegraphics[width=2.4in]{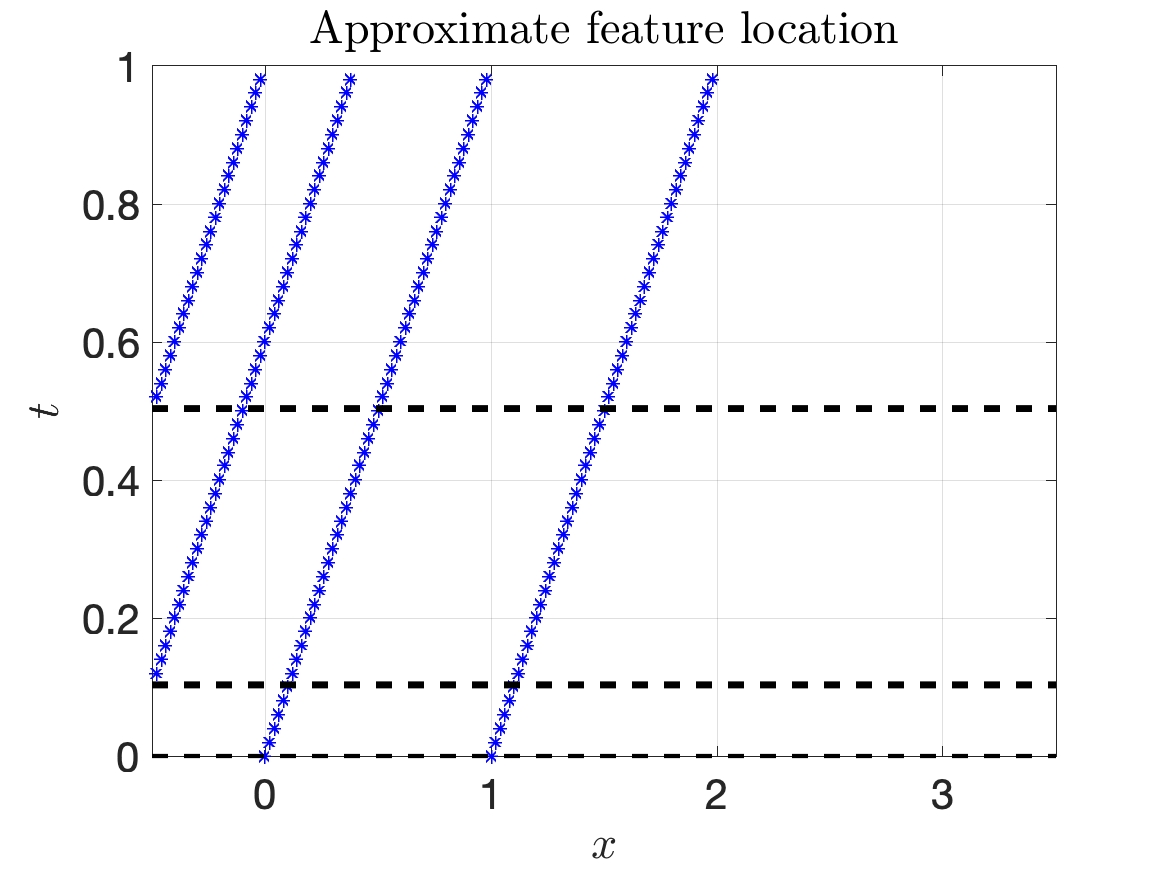} 
\label{test-4 feature loc}
}
\hfill
\subfloat []{
\includegraphics[width=2.4in]{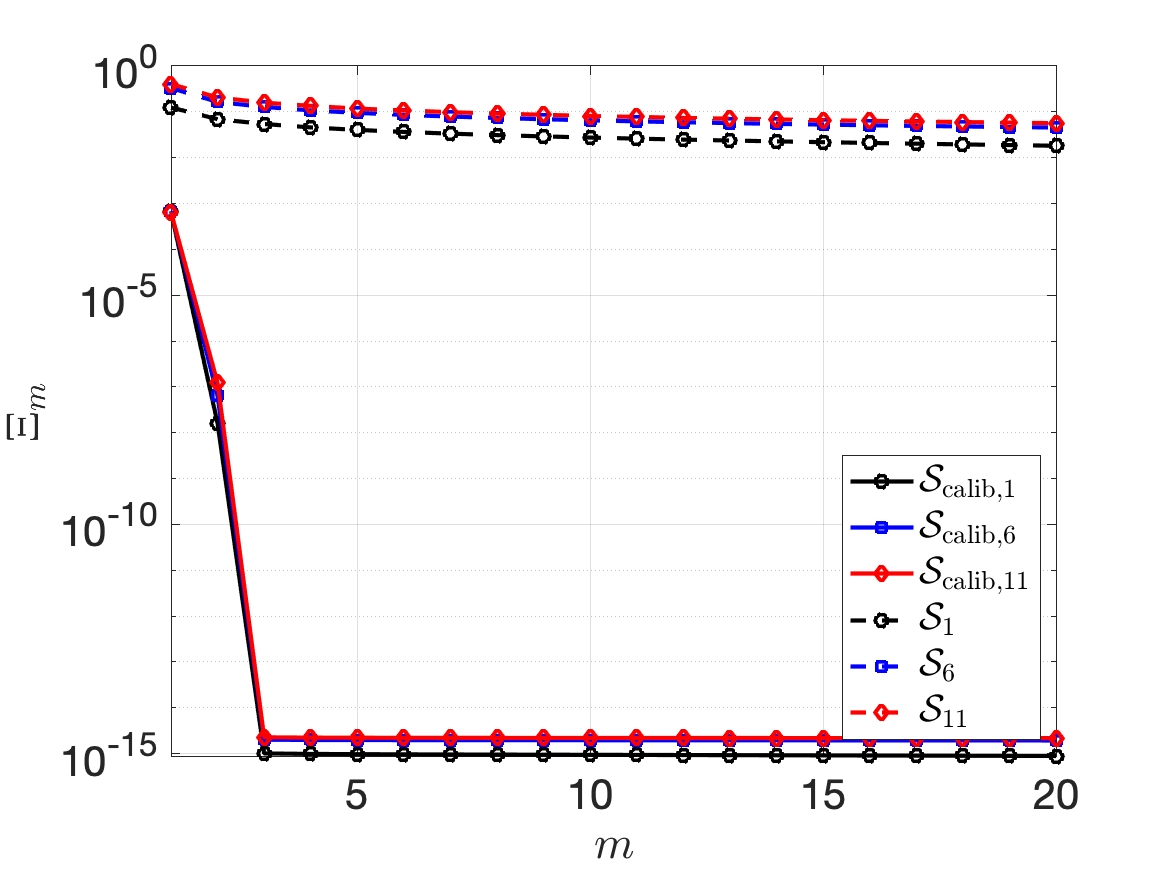} 
\label{test-4 E}
}
\hfill
\caption{\textit{Results for test case-4. (a) Approximate feature location. (b) Compares $\Xi_m(\snapMat_{1/6/11})$ to $\Xi_m(\snapMatCalibSub{1/6/11})$. The y-axis in (b) is on a log-scale. }}	
\end{figure} 

\section{Conclusions}
We have proposed an algorithm to induce a fast singular value decay in a snapshot matrix resulting from hyperbolic equations. The algorithm relies on computing the snapshots on a transformed spatial domain with the transformation computed using feature matching between a reference and the other snapshots. The choice of the reference snapshot ensures that the transformation is a homeomorphism and has a lower and an upper bound on its weak derivative---we found these two properties desirable for both the theoretical analysis and a numerical implementation. To account for feature interaction and formation (i.e., cases where shocks collide, shocks form, etc.), we have proposed an adaptive reference snapshot selection technique. With this technique, we can divide the snapshot matrix into sub-matrices with each sub-matrix containing snapshots with no feature interaction/formation. In each of the sub-matrices, we perform feature matching as usual.  

Under regularity assumptions on the initial data and the flux function, we have proven that feature matching results in a fast $m$-width decay of a so-called calibrated manifold. Our proof exploits the regularity of functions in a calibrated manifold. We have performed numerical experiments on a broad range of problems involving non-linear system of equations and time-dependent boundary conditions. Our experiments verify that feature matching is successful in inducing a fast singular value decay in a snapshot matrix. We also found that feature matching performs exceptionally well for problems where the calibrated manifold contains a single function.   

We observe that although the singular values of a calibrated snapshot matrix decay fast, they can stagnate at a value that scales with the spatial grid resolution. The stagnation is a by-product of computing the spatial transform using a numerical approximation of the exact solution and indicates that, for hyperbolic problems, not much is gained by increasing the dimension of the reduced-order model beyond a certain limit.

\bibliographystyle{abbrv}
\bibliography{../papers}
\appendix

\section{Regularity of functions in $\calibMani{}{\tspace}$}\label{app: proof decay}
The definition of $X_i$ provides $X_i \in C^0(\Omega_i^D)$ by the implicit function theorem and the bound on $\beta_i$. Moreover,
\begin{equation}
\begin{gathered}
D_{t}X_i(x,\tstar) = -\frac{f'(u_0(X_i(x,\tstar)))}{\beta(X_i(x,\tstar),\tstar)} =:\hat{\mcal G}(X_i(t,x),t),\\ D_{x}X_i(x,\tstar) = \frac{1}{\beta(X_i(x,\tstar),\tstar)}=:\tilde{ \mcal G}(X_i(t,x),t). \label{der Xi}
\end{gathered}
\end{equation}
The regularity of $f$ and $u_0$ and the assumption on $\beta$ imply that $\hat{\mcal G}, \tilde {\mcal G} \in C^{\omega-1}(\Omega_i \times D)$ which implies that $X_i \in W^{\omega,\infty}(\Omega_i^D)$ by bootstrapping.


Next, we  show that $\varphi \in L^\infty(\Omega;W^{\omega,\infty}(\tspace))$. Since $\varphi(x,t)\leq \xMax$, we have $\varphi\in L^\infty(\Omega\times\tspace)$. The definition of $\varphi$ in \eqref{def spTransf} implies that $D_t^\omega\varphi\in L^\infty(\Omega\times\tspace)$ if $z\in W^{\omega,\infty}(\tspace)$. When $z$ is a kink location, following the characteristics forwards in time we find 
$z(t) = z(0) + f'(u_0(z_0(0)))\cdot t$, which provides the desired regularity. When $z$ is a shock location, we proceed as follows.

For simplicity, assume that $p_0=1$ with a shock at $z(t)$. The argument remains the same for (non-interacting) multiple shocks. Consider the weak solution
\begin{gather}
u(x,t) = \begin{cases}
\td u_0(x,t)\hsp & \text{in }\Omega^D_0\\
\td u_1(x,t)\hsp & \text{in }\Omega^D_1
\end{cases}.
\end{gather}
Above, $\Omega^D_{0/1}$ are as given in \eqref{def Omegai}. Following the characteristics forward in time, we find 
\begin{gather}
\td u_0(x,t) = u_0(X_0(x,t)),\hspB \td u_1(x,t) = u_0(X_1(x,t)).
\end{gather} 

The assumption on $\beta_i$ means that inside $\Omega_i^D$  characteristics of $u$ are bounded away from intersecting each other. Thus, $\tilde u_0, \tilde u_1$ inherit their regularity from the regularity of the initial data between the features, i.e. $\tilde u_i \in W^{\omega, \infty}(\Omega_i^D)$ and (since intersection of characteristics is not imminent), we can find $c,\epsilon >0$ such that $\td u_0$ has a  extension $\td u_0^{\opn{ex}} \in W^{\omega, \infty}(\Omega_0^{D,\opn{ex}})$ (that is constant along characteristics) with
\begin{gather}
\Omega_0^{D,\opn{ex}}:= \{(x,t)\hsp : \hsp x\leq z(t) + \min(\epsilon,ct),\hsp t\leq T\}. 
\end{gather}
 A similar definition holds for $\td u_1^{\opn{ex}}$. By the Rankine-Hugoniot condition, $z$ satisfies 
\begin{gather}
d_t z(t) = \mcal H(\td u_0^{\opn{ex}}(z(t),t),\td u_1^{\opn{ex}}(z(t),t))\hspB\text{where}\hspB \mcal H(a,b):=\begin{cases}
\frac{f(a)-f(b)}{a-b},\hsp & a\neq b\\
f'(a),\hsp & a =b \label{shock eq}
\end{cases}. 
\end{gather}
Since $f\in C^{\omega + 1}$ we have $\mcal H\in C^{\omega}(\mbb R^2)$ implying that $z$ satisfies $d_t z(t) = h(z(t),t)$ with $h = \mcal H(\td u_0^{\opn{ex}},\td u_1^{\opn{ex}})$ and $h\in C^{\hcoeff-1}\left(\left(\Omega^{D,\opn{ex}}_0\cap \Omega^{D,\opn{ex}}_1\right)\times\tspace\right)$. Since $\Omega^{D,\opn{ex}}_0\cap \Omega^{D,\opn{ex}}_1$ is compact and $\td u_i^{\opn{ex}}$ is Lipschitz, $h$ is globally Lipschitz continuous providing a global solution to \eqref{shock eq}. Furthermore, since $h\in C^{\hcoeff-1}$, $z\in C^{\hcoeff}(\tspace)$. Since $\tspace$ is closed, we have 
$z\in W^{\hcoeff,\infty}(\tspace)$ and thus $\varphi \in L^\infty(\Omega,W^{1,\infty}(\tspace))$.

Using \eqref{trace back} the regularity of $g$ is a direct consequence of the regularity of $u_0, X_i,$ and $\varphi$

\section{Rarefaction fan} \label{app: rarefaction}
Let $X_i(x,\tstar)$ be as given in \eqref{trace back}. We show that the second condition in \eqref{cond trivial} can be satisfied if $\Omega_i$ contains a rarefaction fan. Let $\Omega = (-1,2)$ and let $\tspace = [0,0.5]$ and consider the initial data
\begin{gather}
u_0(x):=\begin{cases}
\left(f'\right)^{-1}\left(0\right),& \hsp x\leq 0\\
\left(f'\right)^{-1}\left(x\right), &\hsp x\in (0,1)\\
\left(f'\right)^{-1}\left(1\right),&\hsp x\in [1,2)
\end{cases}. 
\end{gather}
With the above initial data, the solution reads 
\begin{gather}
u(x,t):=\begin{cases}
0,& \hsp x\leq 0\\
\left(f'\right)^{-1}\left(\frac{x}{t+1}\right), &\hsp x\in (0,1+t)\\
\left(f'\right)^{-1}\left(1\right),&\hsp x\in [1+t,2)
\end{cases}. 
\end{gather}
Assume that for all $t\in \tspace$, $u(\cdot,t)$ has a kink at both $x = 0$ and $x = 1 + t$. Thus, we have two features. The kink locations are given as 
\begin{gather}
z_1(t) = 0,\hspB z_2(t) = 1 + t.
\end{gather}
Using the above relation, for $x\in \Omega_2 = (z_1(t),z_2(t))$, the spatial transform reads
\begin{gather}
\varphi(x,\tstar) = x\left(1 + t\right).
\end{gather}
For $i=2$ and for all $x\in\Omega_2$, the definition of $X_i$ in \eqref{trace back}, the expression for $u_0$, and the above expression for $\varphi$ provides
\begin{gather}
X_2(x,t) + \tstar X_2(x,t) = x\hsp\Rightarrow\hsp X_2(x,\tstar) = \frac{x}{1+t}. 
\end{gather}

\section{Estimate for $\|u\circ\varphi-u\circ\varphi_M\|_{L^2(\Omega\times\tspace)}$} \label{app: proof error feature}
\begin{enumerate}
\item The following proof is an extension of the one given in \cite{Welper2017} for $L^2$ functions. For some $\epsilon > 0$, define $\Omega_{\epsilon}:\{x \in \Omega \hsp :\hsp \opn{dist}(x,\pd\Omega) > \epsilon\}$. Let $u_\epsilon\in C^{\infty}(\Omega_\epsilon)$ be a mollification of $\sol{t}$ over $\Omega_\epsilon$. Then, the following holds
\begin{gather}
\|u_\epsilon-\sol{t} \|_{L^2(\Omega_{\epsilon})}\xrightarrow{\epsilon\to 0} 0,\hspB \|u_\epsilon(\cdot,t)\|_{BV(\Omega_\epsilon)}\leq \| \sol{t}\|_{BV(\Omega_\epsilon)}. \label{app: density}
\end{gather}
Triangle's inequality provides
\begin{equation*}
\begin{aligned}
\|u\circ \varphi-u\circ \varphi_M\|_{L^2(\Omega_\epsilon\times\tspace)} \leq &\|u\circ \varphi-u_\epsilon\circ \varphi\|_{L^2(\Omega\times\tspace)}\\
&+\|u\circ \varphi_M-u_\epsilon\circ \varphi_M\|_{L^2(\Omega_\epsilon\times\tspace)}\\
&+\|u_\epsilon\circ \varphi-u_\epsilon\circ \varphi_M\|_{L^2(\Omega_\epsilon\times\tspace)}. 
\end{aligned}
\end{equation*}
Applying a domain transformation and using \eqref{prop phi}, we find 
\begin{gather}
\|u\circ \varphi-u_\epsilon\circ \varphi\|_{L^2(\Omega\times\tspace)}\lesssim \epsilon,\hspB \|u\circ \varphi_M-u_\epsilon\circ \varphi_M\|_{L^2(\Omega\times\tspace)}\lesssim \epsilon. 
\end{gather}
Because of the above two relations, it is sufficient to bound $\|u_\epsilon\circ \varphi-u_\epsilon\circ \varphi_M\|_{L^2(\Omega_\epsilon\times\tspace)}$. For $s\in [0,1]$, define $\Phi(x,t,s) = s \varphi(x,t) +(1-s) \varphi_M(x,t)$. Using $\Phi$, we write 
\begin{align*}
\|u_\epsilon\circ \varphi-u_\epsilon\circ \varphi_M\|^2_{L^2(\Omega_\epsilon\times\tspace)} = &\int_{\Omega_\epsilon\times\tspace} \left(\int_0^1 \pd_s u_\epsilon (\Phi(x,t,s))ds\right)^2dxdt \\
\leq &\|u_\epsilon\|_{L^\infty(D),BV(\Omega_\epsilon))}\\
&\times \int_{\Omega_\epsilon\times\tspace} \left(\int_0^1 \pd_s |u_\epsilon (\Phi(x,t,s))|ds\right)dxdt \\
\leq &\|u\|_{L^\infty(\tspace;BV(\Omega))} \|u\|_{L^2(\tspace;BV(\Omega))}\\
&\times\|\varphi-\varphi_M\|_{L^\infty(\Omega\times\tspace)}.
\end{align*}
Above, the last inequality follows from \cite{Welper2017} and \eqref{app: density}.
\item

By definition, 
\begin{gather}
\varphi(z_i(0),t) = z_i(t),\hspB \varphi_M(\featureFV{i}(0),t) = \featureFV{i}(t).
\end{gather}
We refer to $z_i(0)$ and $\featureFV{i}(0)$ as the nodes and to $z_i(t)$ and $\featureFV{i}(t)$ as the node values of a spatial transform.
We introduce an intermediate (continuous and piecewise linear) spatial transform $\hat\varphi$ that has the same nodes as $\spTransf{t}{t}$ and the same nodal values as $\spTransfGrid{t}{t}$ i.e., $\hat\varphi(z_i(0),t) = \featureFV{i}(t)$.
By triangle's inequality,
\begin{gather}
\|\varphi_M-\varphi\|_{L^{\infty}(\Omega\times \tspace)}\leq \|\varphi-\hat\varphi\|_{L^{\infty}(\Omega\times\tspace)} + \|\hat\varphi-\varphi_M\|_{L^{\infty}(\Omega\times\tspace)}.
\end{gather}
Because $\varphi$ and $\hat\varphi$ have the same nodes, we conclude that 
\begin{gather}
\|\varphi-\hat\varphi\|_{L^{\infty}(\Omega\times\tspace)} = \max_{j}\|\featureFV{j}-z_j\|_{L^{\infty}(\tspace)}.
\end{gather}
It is easy to check that the maximum of $|\hat\varphi(\cdot,t)-\varphi_M(\cdot,t)|$ occurs at either the nodes $\{z_i(0)\}_i$ or $\{\featureFV{i}(0)\}_i$. Computing $|\hat\varphi(\cdot,t)-\varphi_M(\cdot,t)|$ at these nodes provides
\begin{equation}
\begin{aligned}
 \|\hat\varphi(\cdot,t)-\varphi_M(\cdot,t)\|_{L^{\infty}(\Omega)}\leq  &\|D_x\spTransfGrid{.}{t}\|_{L^{\infty}(\Omega)}\max_{j}|\featureFV{j}(t)-z_j(t)|\\
 \leq &\mcal K_1\max_{j}|\featureFV{j}(t)-z_j(t)|.
\end{aligned}
\end{equation}
where $\mcal K_1$ is the constant in \eqref{prop phi}. 
\end{enumerate}

\section{Relation to MRA}\label{app: relate MRA}
We briefly relate our feature detection method to that proposed in \cite{MRAdetect2014}. We specialise the formulation for a FV scheme, generalisations to arbitrary order discontinuous-Galerkin type schemes can be found in the references therein.
We divide $\Omega$ into uniform $N_l=2^l$ elements with $l\in\mbb N$. Such a choice of $N_l$ results in a hierarchy of grids parameterised by $l$. With $\mcal I^l_i$ we represent the $i$-th cell at level $l$. With $u^l_i(t)$ we denote the FV approximation of $\sol{t}$ in $\mcal I_i^l$. 

In the middle of every $\mcal I_i^{l-1}$ lies a face that is shared between $\mcal I^l_{2i-1}$ and $\mcal I^{l}_{2i}$. Let $J^{l-1}_i(t)$ denote the jump of the FV solution across this face i.e.,
\begin{gather}
    J_i^{l-1}(t) = |u_{2i-1}^l(t)-u_{2i}^l(t)|.\label{hierarchy d}
\end{gather}
Thus, given $u^l_i$, we can compute all of $J_i^{l-1}$. The coefficient $J_i^{l-1}/2$ is the same as the so-called wavelet coefficient in the MRA. Define
\begin{gather}
    D^{l-1}(t):=\max_{i\in 1,\dots,2^{l-1}}J_i^{l-1}(t).
\end{gather}
Similar to $\mcal B(t)$ in \eqref{set B}, define
\begin{gather}
    \mcal B^{l-1}(t) := \{i\hsp : \hsp |J_i^{l-1}(t)| > C\times D^{l-1}(t),\hsp i\in\{1,\dots 2^{l-1}\} \}.
\end{gather}
At level $l-1$, cells with index in $\mcal B^{l-1}$ are flagged. Due to the grid hierarchy, the cells at level $l$ that have a discontinuity are $\{2i-1\hsp :\hsp i\in\mcal B^{l-1} \}$ and $\{2i\hsp :\hsp i\in\mcal B^{l-1} \}$. Above, $C$ is the same as that defined in \eqref{set B}.

As is clear from the definition of $J_e^{l-1}$, in MRA one computes the jump in the FV solution at every alternate face.  Equivalently, MRA does not compute jumps at any face at level $l-1$. Therefore, a discontinuity (independent of its strength) aligned with any of these faces is not detected. 
Such discontinuities do not contribute to an oscillatory numerical solution.  Therefore,  for the purpose of flagging cells for suppressing oscillations, MRA is sufficient.  However, in the present context, missing out on large shocks is undesirable. Therefore, we compute the jumps at all the faces, which allows us to detect shocks that could be aligned with cell boundaries.

\section{Flagging of discontinuous regions} \label{app: flag regions}
For simplicity, we assume that $\solFV{t}$ is a projection of $\sol{t}$ onto the FV basis. At least computationally, for a small enough grid size, similar observation holds for a $\solFV{t}$ computed with a FV scheme. 

\begin{enumerate}
    \item Locally differentiable: If $\sol{t}|_{\mcal I_{e-1}\cup \mcal I_{e}}$ is $C^1$ then Taylor expansion provides
\begin{gather}
J_e \leq \Delta x \|\pd_x\sol{t}\|_{C^0(\mcal I_{e-1}\cup \mcal I_{e})}.
\end{gather}
\item Discontinuous: Let $\sol{t}$ have a discontinuity inside $\mcal I_{e}$. Let the point of discontinuity be $z^D = x_{e} + l\times\Delta x$ where $l\in (0,1)$. Furthermore, let $\sol{t}$ be piecewise constant in $\mcal I_{e-1} \cup \mcal I_{e}$ with the value before and after the discontinuity being $u_-$ and $u_+$, respectively. Then
\begin{gather}
J_e = |u_--u_+|(1-l). \label{shock wavelet}
\end{gather}
\item Kink: Assume that $\sol{t}$ is continuous, is piecewise linear in $\mcal I_{e-1}\cup \mcal I_{e}$ and has a kink at $z^K = x_{e} + l\times\Delta x$. Then, assuming $u(z^K,t) = 0$, $\sol{t}|_{\mcal I_{e-1}\cup \mcal I_{e}}$ reads
\begin{gather}
    \sol{t}|_{\mcal I_{e-1}\cup \mcal I_{e}} = \begin{cases}
    (x-x_K) \pd u_- &x < z_K\\
    (x-x_K)\pd u_+&x \geq z_K
    \end{cases}
\end{gather}
Above, $\pd u_-$ and $\pd u_+$ are the left and right slopes respectively. With the above $\sol{t}$, we find
\begin{gather*}
    J_e = \frac{\Delta x}{2}|(\pd u_--\pd u_+)l^2-2\times \pd u_+|.
\end{gather*}
\end{enumerate}

With the above relations and the form of $\mcal B(t)$ given in \eqref{set B}, we draw the following three conclusions. First, regions where the solution is $C^1$ but has a large gradient might be identified as discontinuities. Second, shocks with a strength (i.e., $|u_--u_+|$) of $\mcal O(\Delta x)$ might go undetected. Third, kinks with a large left and right derivative might be identified as discontinuities. In relation to the second point, in case $J_e(t) < C\Delta x$, where $C$ is as given in \eqref{set B}, one can show that the semi-discrete numerical solution already has the regularity necessary for a fast $m$-width decay. 

\end{document}